\documentclass[a4paper, 10pt, reqno]{amsart}

\evensidemargin
\oddsidemargin
\usepackage{amsmath, amsthm, latexsym, amssymb, bbm, url}
\usepackage[dvips]{color}

\usepackage{calc,  subfigure, bm}
\usepackage[dvips]{graphicx}

\newcommand{\e}{\varepsilon}

\renewcommand{\phi}{\varphi}

\newcommand{\ssup}[1] {{\scriptscriptstyle{({#1}})}}

\newcommand{\one}{{\mathbf 1}}

\newcommand{\Prob}{{\mathrm{Prob}}}

\newcommand{\R}{\mathbb R}
\newcommand{\Z}{\mathbb Z}

\newcommand{\N}{\mathbb N}

\newcommand{\E}{\mathbb E}
\renewcommand{\P}{\mathbb P}

\newcommand{\emptypp}{\varnothing}

\newtheorem{theorem}{Theorem}[section]
\newtheorem{lemma}[theorem]{Lemma}
\newtheorem{corollary}[theorem]{Corollary}
\newtheorem{prop}[theorem]{Proposition}

\newcounter{remnr}
\newenvironment{remark}{\refstepcounter{remnr}
{\sf Remark~\arabic{remnr}.\ }\nopagebreak  }%
{\nopagebreak {\hfill{$\diamond$}}\\ }

\renewcommand{\phi}{\varphi}

\renewcommand{\P}{\mathbb{P}}
\renewcommand{\E}{\mathbb{E}}

\newcommand{\Zepm}{Z^{\ssup{e\pm}}_t}
\newcommand{\Za}{Z^{\ssup 1}_t}
\newcommand{\Zb}{Z^{\ssup 2}_t}

\usepackage[backend=bibtex,doi=false,isbn=false,url=false,firstinits=true]{biblatex}
\renewbibmacro{in:}{}
\bibliography{twosite}
\usepackage{hyperref}
\usepackage{enumerate}
\numberwithin{equation}{section}

 \usepackage[foot]{amsaddr}
 \usepackage{cases}

\begin{document}
\title[A new phase transition in the PAM with partially duplicated potential]{A new phase transition in the parabolic Anderson model \\ with partially duplicated potential}
\author{Stephen Muirhead$^1$}
\address{$^1$Mathematical Institute, University of Oxford (current address: School of Mathematical Sciences, Queen Mary University of London)}
\email{s.muirhead@qmul.ac.uk}
\author{Richard Pymar$^2$}
\address{$^2$Department of Economics, Mathematics and Statistics, Birkbeck, University of London}
\email{r.pymar@bbk.ac.uk}
\author{Nadia Sidorova$^3$}
\address{$^3$Department of Mathematics, University College London}
\email{n.sidorova@ucl.ac.uk}
\subjclass[2010]{60H25 (Primary) 82C44, 60F10 (Secondary)}
\keywords{Parabolic Anderson model, localisation}
\thanks{The first author was supported by the Engineering \& Physical Sciences Research Council (EPSRC) Fellowship EP/M002896/1 held by Dmitry Belyaev. The second author was supported by the EPSRC Grant EP/M027694/1 held by Codina Cotar.}
\date{\today}
\begin{abstract}
We investigate a variant of the parabolic Anderson model, introduced in previous work, in which an i.i.d.\! potential is partially duplicated in a symmetric way about the origin, with each potential value duplicated independently with a certain probability. In previous work we established a phase transition for this model on the integers in the case of Pareto distributed potential with parameter $\alpha > 1$ and fixed duplication probability $p \in (0, 1)$: if $\alpha \ge 2$ the model completely localises, whereas if $\alpha \in (1, 2)$ the model may localise on two sites. In this paper we prove a new phase transition in the case that $\alpha \ge 2$ is fixed but the duplication probability $p(n)$ varies with the distance from the origin. We identify a critical scale $p(n) \to 1$, depending on~$\alpha$, below which the model completely localises and above which the model localises on exactly two sites. We further establish the behaviour of the model in the critical regime.
\end{abstract}

\maketitle

\section{Introduction}
\subsection{The parabolic Anderson model with partially duplicated potential}

Given a potential field $\xi : \Z^d \to \R$, the parabolic Anderson model (PAM) is the solution to the Cauchy problem with localised initial condition 
\begin{align}
\label{eq:PAM} \partial_t u(t,z) &=\Delta u(t,z)+\xi(z)u(t,z), & (t,z)\in (0,\infty)\times \Z^d,\\
\nonumber u(0,z)&=\one_{\{0\}}(z), & z\in\Z^d,
\end{align} 
where $\Delta$ is the discrete Laplacian acting on functions $f:\Z^d\to\R$ by
\begin{align*}
(\Delta f)(z)=\sum_{|y-z|=1}(f(y)-f(z)), \qquad z\in\Z^d, 
\end{align*}
with $|\cdot|$ the standard $\ell_1$ distance. 
\smallskip

The PAM models the interaction between two competing forces: a \emph{smoothing effect} coming from the Laplacian, and a \emph{roughening effect} coming from irregularities in the potential. It is well-known that the PAM exhibits certain intermittent phenomena \cite{GM90, Konig16}, for example if the potential is sufficiently inhomogeneous the solution will tend to concentrate, at typical large times, on a small number of spatially-disjoint clusters of sites \cite{FM, GKM07, KLMS, ST}. If the potential is a random field, the strongest form of this phenomenon is known as \emph{complete localisation}, in which there exists a $\mathbb{Z}^d$-valued process $Z_t$ such that, as $t\to\infty$, 
\begin{align}
\label{e:cl}
 \frac{u(t,Z_t)}{\sum_{z\in\mathbb{Z}^d}u(t,z)}\to1\quad\mbox{in probability.}
 \end{align}

In \cite{MPS} we introduced a variant of the PAM in which the potential is partially duplicated in a symmetric way about the origin; our motivation was to investigate the kind of duplication of a strongly inhomogeneous potential that could cause complete localisation to fail. To the best of our knowledge, this was the first study of the complete localisation phenomenon in the PAM for a potential that was not independent. We recall this model now.
\smallskip

We restrict our attention to the case $d = 1$, and define an auxiliary random field $\xi_0: \Z \to [1, \infty)$ consisting of independent Pareto random variables with parameter $\alpha>1$, that is, with distribution function 
\begin{align*}
F(x)=1-x^{-\alpha}, \quad x \ge 1.
\end{align*}
Abbreviate $\N_0 = \N\cup\{0\}$, and define a random field $\xi : \Z \to [1, \infty)$ by setting 
$\xi(n) =\xi_0(n)$ for each $n\in\N_0$ and, for each $n\in \N$, independently setting
\begin{equation}
\label{eq:pq}
\xi(-n) = \begin{cases}
\xi_0(n)&\mbox{with probability $p(n)$},\\
\xi_0(-n)&\mbox{with probability $q(n)$},
\end{cases}
\end{equation}
where $p:\mathbb{N}\to[0,1]$ are pre-determined duplication probabilities and $q(n)=1-p(n)$. The model we consider is the solution to \eqref{eq:PAM} with partially duplicated potential $\xi$. If $p\equiv 0$ we recover the PAM with i.i.d.\ Pareto distributed potential for which it is known \cite{KLMS} that complete localisation holds (note that the restriction $\alpha > 1$ ensures that the solution exists, see \cite{GM90}). On the other hand, if $p\equiv1$ then a simple symmetry argument shows that complete localisation fails.
\smallskip

In \cite{MPS} we analysed the model for fixed duplication probabilities $p(n)\equiv p \in(0,1)$.  We discovered a phase transition in the parameter $\alpha$: if $\alpha\ge2$ the PAM exhibits complete localisation for any $p \in (0,1)$, whereas if $\alpha \in (1, 2)$ there is a positive probability (depending on $p$) that the model has non-negligible mass on exactly two sites at typical large times. The fact that the critical value of $\alpha = 2$ does not depend on the value of $p$ surprised us, and led naturally to the question we consider in this paper:
\begin{quote}
\textit{In the case $\alpha \ge 2$, might complete localisation fail if we allow $p(n)$ to vary with $n$?}
\end{quote}

It is easy to guess a reasonable answer to this question: since fixed $p(n) \equiv p \in(0,1)$ results in complete localisation, and for $p \equiv 1$ complete localisation fails, it is natural to expect that complete localisation will fail if $p(n) \to 1$ sufficiently quickly. This intuition turns out to be correct, and our main result identifies precisely a critical scale $p(n) \to 1$, depending on $\alpha$, below which the model completely localises and above which the model localises on exactly two sites with overwhelming probability; in contrast to the case of fixed $p(n) \equiv p \in (0, 1)$, these outcomes happen with overwhelming probability. We complete the picture by analysing the behaviour of the model for $p(n) \to 1$ \emph{at} the critical scale. 
\smallskip

For simplicity, we have decided to omit the case $\alpha \in (1, 2)$ from our results. Although this case can also be treated with our methods, it is less interesting because it lacks a critical scale. Indeed for any $p(n) \to 1$ the model localises on exactly two sites with overwhelming probability; this is completely natural in view of our results in \cite{MPS}. We have also chosen not to treat the case $d \ge 2$; there are major additional technicalities present in higher dimensions (see the comments in \cite{MPS} for more detail), and we leave this for future work.

\subsection{The phase transition in the model}\label{subsecphase}
We now state our results more formally. For the remainder of the paper we assume $\alpha\ge2$. Let $p:\N\to [0,1]$ be an eventually increasing sequence such that $p(n)\to 1$ as $n\to \infty$, and define $q(n)=1-p(n)$. We denote by the same symbols the functions $p, q: \mathbb{R}_+ \to [0, 1]$ defined by $p(x) = p(\lceil x \rceil)$ and $q(x) = q(\lceil x \rceil)$ (and similarly for all functions defined on $\N$ in the sequel).\smallskip

Let $\xi$ be defined as at \eqref{eq:pq} using the sequence of duplication probabilities $p(n)$. We henceforth refer to $\xi$ as the potential, and denote its corresponding probability and expectation by $\Prob$ and $\mathrm{E}$ respectively. \smallskip

It follows from~\cite{GM90} by the same argument as in the i.i.d.\! case that the solution to \eqref{eq:PAM} exists since $\alpha>1$, and is given by the Feynman-Kac formula
\begin{align*}
u(t,z)=\E\Big[\exp\Big\{\int_0^t\xi(X_s)ds\Big\}\one\{X_t=z\}\Big], \qquad (t,z)\in (0,\infty)\times\Z,
\end{align*}
where $(X_t)_{t\ge 0}$ is a continuous-time random walk on $\Z$ with generator $\Delta$ started at the origin and $\P$ and $\E$ are its corresponding probability and expectation. We denote by 
\begin{align*}
U(t) = \sum_{z\in\Z}u(t,z)
\end{align*}
the total mass of the solution.
\smallskip

Let $D=\{z\in \N_0 : \xi(z)=\xi(-z)\}$ denote the set of positive integers whose potential values are \emph{duplicated}, and $E=\N_0\backslash D$ the set of integers whose potential values are unique (or \emph{exclusive}) to them. Note that this notation differs slightly to \cite{MPS}, where $D$ and $E$ were subsets of $\Z$ rather than~$\mathbb{N}_0$. For each $t>0$, define the functional $\Psi_t:\Z\to \R$ by 
\begin{align*}
\Psi_t(z)=\xi(z)-\frac{|z|}{t}\log \xi(z). 
\end{align*}
Let $\Za\in D$ be a maximiser of $\Psi_t$ over $D$; the proof of the existence of $\Za$ is standard since $\alpha > 1$ (see for instance \cite[Lem.\ 3.2]{MPS}). Again this notation differs slightly to \cite{MPS}, where $\Za$ was defined as a maximiser of $\Psi_t$ over $\N$. 
\smallskip

Our first result is to establish, in all regimes, that the solution of the model eventually concentrates, at typical large times, on the sites $\Za$ and $-\Za$. This result is similar to \cite[Thm.\ 1.2]{MPS}, and actually holds for any $p(n) \in [0, 1]$ and $\alpha > 1$; to avoid complications we treat formally only the case of eventually increasing $p(n) \to 1$ and $\alpha \ge 2$.

\begin{theorem}[Localisation of the model]
\label{t:main0}
As $t\to\infty$,
\begin{align*}
\frac{1}{U(t)}\big[u(t,\Za)+u(t,-\Za)\big]\to 1\qquad\text{in probability}.
\end{align*}
\end{theorem}

The next series of results -- the main results of the paper -- establish the phase transition in the model, identifying the critical scale $p(n) \to 1$ below which the model completely localises (i.e.\ localisation on one site) and above which the model localises on exactly two sites.
 \smallskip
It is most convenient to define the critical scale $p(n) \to 1$ by reference to the rate of growth of the increasing function
\begin{align}
\label{e:eta}
\eta(n)  = \sum_{z=1}^n q(z) =  \int_0^n  q(x) \, dx , \quad n \in \mathbb{N}_0,
\end{align}
which is the expected value of the counting function $N(n)$ defined as 
\[
 N(n)=\sum_{z=1}^{n}\one\{z\in E\}.
\] The critical scale for $\eta(n)$ is then given by the function $\kappa:\mathbb{N}\to\mathbb{R}^+$
\begin{align}
\label{e:kappa}
\kappa(n):=\begin{cases}
n^{2/\alpha}&\mbox{if $\alpha>2$},\\
\frac{n}{\log n}&\mbox{if $\alpha=2$}.
\end{cases}
\end{align}
Here and in the sequel we use $f(x) = o(g(x))$ and $f(x) \ll g(x)$ interchangeably to denote that $f(x)/g(x) \to 0$ as $x \to \infty$. Similarly, we use $f(x) \sim g(x)$ to denote that $f(x)/g(x) \to 1$ as $x \to \infty$.

\begin{theorem}[Subcritical regime]
\label{t:main1}
Let $\eta(n)\ll \kappa(n)$. As $t\to\infty$,
\begin{align*}
\frac{u(t,\Za)}{u(t,-\Za)}\to 1\qquad\text{in probability.}
\end{align*}
\end{theorem}

\begin{theorem}[Supercritical regime]
\label{t:main2}
Let $\eta(n)\gg \kappa(n)$. As $t\to\infty$,
\begin{align*}
\Big|\log\frac{u(t,\Za)}{u(t,-\Za)}\Big|\to \infty\qquad\text{in probability.}
\end{align*}
\end{theorem}

Remark that, in the supercritical regime, Theorems \ref{t:main0} and \ref{t:main2} together imply that the model completely localises  (i.e.\ equation \eqref{e:cl} holds), whereas in the subcritical regime, Theorems \ref{t:main0} and \ref{t:main1} together imply that the model localises on exactly two sites (note that $\Za \neq 0$ eventually almost surely by Lemma \ref{l:0} below). Indeed in the supercritical regime we actually show more, namely that the solution is, for large times, distributed approximately evenly across two sites. \smallskip

Our final result establishes the behaviour of the model in the critical regime. In this regime the model also localises on two sites, but unlike in the subcritical regime the amount of mass on each site is random; we exhibit a limit theorem for this behaviour.

\begin{theorem}[Critical regime]
\label{t:main3}
Let $\eta(n)\sim \beta \kappa(n)$ with $\beta>0$. As $t\to\infty$,
\begin{align*}
\frac{u(t,\Za)}{u(t,-\Za)}\Rightarrow \exp\big\{\sqrt{2\beta}\sigma BN\big\},
\end{align*}
where $\Rightarrow$ denotes weak convergence, $\sigma$ is the positive constant defined by
\[ \sigma^2 :=\begin{cases}
\frac{\alpha}{(\alpha-2)(\alpha-1)^2}&\mbox{if }\alpha>2, \\
1&\mbox{if }\alpha=2,
\end{cases} \]
$B$ is a random variable defined through the weak limit 
\begin{align*}
B\stackrel{\mathcal{L}}{=}\lim_{t\to\infty}\frac{(\Za)^{1/\alpha}}{\xi(\Za)}
\end{align*}
which exists and is strictly positive almost surely, and $N$ is an independent standard normal random variable. In Proposition~\ref{LB1} we give the explicit law of the random
variable $B$.
\end{theorem}

\begin{remark}
The condition that $p(n) \to 1$ is eventually increasing is mostly technical, and it could be replaced by other regularity assumptions without significant change to the results. In the non-critical regimes we use this condition only in the proof of Lemma~\ref{l:etaeta}, and here it would be sufficient to assume instead that $\eta$ is regularly varying. In the critical regime we additionally use this condition in the proof Lemma \ref{1101}, which allows us to deduce the asymptotic behaviour of $q(n)$ from that of $\eta(n)$. Without this condition we would need to assume, rather than deduce, that $q(n)$ satisfies this asymptotic behaviour. \end{remark}

\begin{remark}
Observe that, in the case $\alpha > 2$, the constant $\sigma^2$ in Theorem \ref{t:main3} is the variance of $\xi(0)$, a Pareto random variable with parameter $\alpha$. Naturally, this constant approaches infinity as $\alpha \downarrow 2$, and in the case $\alpha = 2$ we instead need to change the scale $\kappa(n)$ by a logarithmic factor in order to get a non-trivial limit for the ratio $u(t, \Za)/u(t, -\Za)$.
\end{remark}

\begin{remark}
As in \cite{MPS}, our results can be recast as a demonstration of the robustness, or lack thereof, of the total mass of the solution of the PAM with i.i.d.\ potential under a resampling of some of the potential values.  More precisely, suppose $u(t, z)$ denotes the solution of the PAM on $\mathbb{Z}$ with the i.i.d.\ potential~$\xi_0$, with $U(t) = \sum_z u(t, z)$ the total mass of the solution. Now resample each potential value $\xi(z)$ independently with probability $q(|z|)$, and let $\tilde u(t, z)$ be the solution of the PAM with this resampled potential, with $\tilde{U}(t) = \sum_z \tilde{u}(t, z)$ the total mass of the solution. Defining $\eta$ and $\kappa$ as at~\eqref{e:eta} and~\eqref{e:kappa}, our methods, suitably translated, demonstrate the following phase transition. If $\eta(n) \ll \kappa(n) $, then $U(t)/\tilde U(t) \to 1$ in probability. By contrast, if $\eta(n) \gg \kappa(n)$, then $|\log  U(t)/\tilde U(t)| \to\infty$ in probability.
\end{remark}

\subsection{Heuristics for the critical scale}
\label{sec:heur}
It is well-known \cite{KLMS} that the solution to the PAM with i.i.d.\ Pareto potential is sharply peaked, with the peaks having first-order approximation
\[ \log u(t,z)\approx t\Psi(t)=t\xi(z)-|z|\log\xi(z). \]
As discussed in \cite{MPS}, the PAM with partially duplicated potential is difficult to analyse because there can be two sites which maximise this functional, and so in order to distinguish these we must turn to the second-order contributions; since these depend on the potential values along paths to $\Za$ and $-\Za$, they are challenging to understand. Although second-order contributions were also studied in \cite{MPS}, our results require a significantly finer analysis, particularly in the critical regime, which adds additional complications. \smallskip

We now present heuristics demonstrating how the second-order contributions give rise to the phase transition in Theorems \ref{t:main1} and~\ref{t:main2}. \smallskip

The Feynman-Kac formula allows us to decompose the total solution $U(t)$ into contributions from each geometric path from the origin. In \cite{MPS} we showed that in the case $\alpha<2$ it is sufficient to consider only the direct paths to $\Za$ and $-\Za$, in the sense that the contribution to $U(t)$ from all other paths is negligible in comparison. On the contrary, for $\alpha\ge2$ we believe this to no longer be the case, and we must consider paths with \emph{loops} in order to capture the dominant contribution. 
\smallskip

Despite this, we show in Proposition~\ref{p:pq} that, in the subcritical and critical regimes at least, we can adequately control these loops so as to represent $u(t,\Za)/u(t,-\Za)$ as a sum over all non-duplicated sites, namely by writing
\begin{align}
\label{heur1}
\log \frac{u(t,\Za)}{u(t,-\Za)}\approx\sum_{z\in E \cap[1,\Za] }\left[\log\left(1-\frac{\xi(-z)}{\xi(\Za)}\right)-\log\left(1-\frac{\xi(z)}{\xi(\Za)}\right)\right].
\end{align}
Formalising \eqref{heur1} is one of the main obstacles to overcome in this paper, and it is considerably more challenging than the equivalent statement in \cite{MPS}. Indeed our argument breaks down entirely in the supercritical regime, and a modification of the technique is needed; this will be discussed in more detail in Section \ref{sec:outline}.
\smallskip

Focusing on the subcritical and critical regimes, we give heuristics for the scale of \eqref{heur1}, showing how its behaviour is determined by whether $\eta(n) \ll \kappa(n)$. If we make the (unjustified, but revealing) assumption that we can perform a Taylor expansion, \eqref{heur1} becomes
\begin{equation}
\label{heur2}
\log\frac{u(t,\Za)}{u(t,-\Za)}\approx \xi(\Za)^{-1}\sum_{z\in E \cap[1,\Za]}(\xi(z)-\xi(-z)).
\end{equation}
In the case $\alpha>2$, we can apply a central limit theorem (again this is unjustified, but revealing): as there are approximately $\eta(\Za)$ sites in $E \cap[1,\Za]$, and by the symmetry of the model, the fluctuations in the sums in \eqref{heur2} are of order $\eta(\Za)^{1/2}$. Since we prove in Proposition~\ref{LB1} that 
\[ \xi(\Za)\approx (t/\log t)^{1/(\alpha-1)}\quad\mbox{and}\quad\Za\approx(t/\log t)^{\alpha/(\alpha-1)}, \]
the right-hand side of \eqref{heur2} is of order
\[  {\Za}^{-1/\alpha}  \eta(\Za)^{1/2} . \]
Notice that this expression is $o(1)$ if and only if $\eta(n)\ll\kappa(n)$, and hence we conclude that \eqref{heur2} is $o(1)$ in the subcritical regime, and of finite order in the critical regime.   \smallskip

In the case $\alpha=2$, the fluctuations in the sums in \eqref{heur2} are instead of order $\eta(\Za)^{1/2}(\log \eta(\Za))^{1/2}$, and so the right-hand side of \eqref{heur2} is of order
\[ {\Za}^{-1/2}  \eta(\Za)^{1/2} (\log \eta(\Za))^{1/2} .\]
Again this expression is $o(1)$ if and only if $\eta(n)\ll\kappa(n)$, and so we reach the same conclusion.


\bigskip
\section{Outline of the proof}
\label{sec:outline}

In this section we give an outline of the proof of our main results. Along the way we introduce four key intermediate statements, and complete the proofs of Theorems \ref{t:main0}--\ref{t:main3} assuming these statements.

\subsection{Outline of the key steps}

The main steps in the proof are similar to those undertaken in \cite{MPS}, but with certain additional complications. In particular, our technique in Step $2$ below is significantly more involved than its equivalent in \cite{MPS}; this is explained in more detail below.
\smallskip

\textbf{Step $1$: Trimming the path set.}
As already remarked, the Feynman-Kac formula allows us to decompose the total solution $U(t)$ based on contributions from geometric paths from the origin. The first step is to eliminate paths that \textit{a priori} make a negligible contribution to the solution, either because they do not end at $\{-\Za, \Za\}$ or because they make too many jumps. This step is rather similar to in \cite{MPS}, and here we have streamlined the approach.
\smallskip

To define the \textit{a priori} negligible paths, we first introduce the scales 
\begin{align*}
r_t=\Big(\frac{t}{\log t}\Big)^{\frac{\alpha}{\alpha-1}}
\qquad\text{and}\qquad
a_t=\Big(\frac{t}{\log t}\Big)^{\frac{1}{\alpha-1}}
\end{align*}
which, as formalised in Proposition \ref{LB1}, give the asymptotic order of $\Za$ and $\xi(\Za)$ respectively. For technical reasons, we also introduce some auxiliary positive monotone scaling functions $f_t \to 0$ and $g_t \to \infty$ which can be thought of as being arbitrarily slowly decaying or growing. We shall need these scales to satisfy
\begin{align}
\label{fg1}
g_t,1/f_t \ll \log\log t
\end{align} 
as well as
\begin{subnumcases} {\log g_t, \log(1/f_t)\ll }
\log \eta(r_t) &  \label{fg2a} \\
\log \frac{r_t}{\eta(r_t)} & \label{fg2b} \\
\Big|\log\frac{\eta(r_t)}{\kappa(r_t)}\Big|, & \label{fg2c}
\end{subnumcases}
unless the functions on the right-hand side do not tend to infinity (which happens to $\log\eta(r_t)$ in the regime when $\eta(n)$ converges and to $\big|\log (\eta(r_t) / \kappa(r_t)) \big|$ in the critical regime; observe that $\log (r_t /\eta(r_t))$ always tends to infinity due to $q(n)\to 0$). These requirements guarantee that any power of $g_t$ and~$f_t$ will be slower growing or decaying than any powers of the functions under the logarithms on the right-hand side. Observe that \eqref{fg2a} and \eqref{fg2b} allows us to separate $\eta$ from its lower and upper limits (the case of bounded $\eta$ we treat somewhat separately), whereas \eqref{fg2c} allows us, in the non-critical regimes, to separate $\eta$ from the critical scale.
\smallskip

Define $R_t= \Za(1+f_t)$ and let $J_t$ be the number of jumps of $(X_s)$ by time $t$. We decompose the total mass $U(t)$ into a significant component
\[U_0(t) = \E \Big[\exp\Big\{\int_0^t \xi(X_s)ds\Big\} \one\{J_t \le R_t, X_t \in \{-\Za, \Za\}   \} \Big] \]
and a negligible component $U_1(t) = U(t) - U_0(t)$. 
\smallskip

In Section~\ref{sec:prelim} we prove that $U_1$ is negligible with respect to~$U$ as long as certain typical properties of~$\xi$ hold. To define these properties, let $\Zb$ be a maximiser of $\Psi_t$ on the set $D \setminus \{\Za\}$ and let $\Zepm$ be  a maximiser of $\Psi_t$ on the set $-E\cup E$. The proof of the existence of $\Zb$ and $\Zepm$ is standard. The typical properties are contained in the event
\begin{align*}
\mathcal{E}^1_t
&=\Big\{f_t<\frac{\Za}{r_t}<g_t, \, f_t<\frac{\xi(\Za)}{a_t}<g_t, \, \frac{\Psi_t(\Za)-\Psi_t(\Zb)}{a_t}>f_t , \Psi_t(\Zepm)<\Psi_t(\Zb), \\
&\phantom{aaaaaa}
\frac{\xi(\Za)-\xi(z)}{a_t} > f_t \ \forall \, |z| \in [0, R_t] \setminus \Za, \, \xi(z) < \frac{|z|}{t} \log \frac{|z|}{2 e t} \  \forall \, |z| > r_t g_t  \Big\},
\end{align*}
which in particular guarantees a large gap between the value of $\Psi_t$ at sites in $\{-\Za, \Za\}$ and all other sites. Our main conclusion in this step is summarised by the following proposition.

\begin{prop} 
\label{u12}
Almost surely, 
\begin{align*}
\frac{U_1(t)}{U(t)}\one_{\mathcal{E}^1_t} \to 0
\end{align*}
as $t\to\infty$.
\end{prop}

\textbf{Step $2$: Representing the ratio $u(t, \Za)/u(t, -\Za)$ as a sum over non-duplicated sites.}
The next step is to formalise \eqref{heur1} representing the ratio $u(t, \Za)/u(t, -\Za)$ as a sum over the non-duplicated sites between $-\Za$ and~$\Za$. As mentioned in Section \ref{sec:heur}, this equation is only valid in the subcritical and critical regimes; in the supercritical regime we show instead that the ratio $u(t, \Za)/u(t, -\Za)$ can be decomposed into a sum over a \emph{subset} of the non-duplicated sites, and another term that we are able to control. Our choice of this subset will turn out to be extremely delicate.
\smallskip

To make this precise we introduce some more notation. Define a threshold function
\begin{align*}
\theta_t=\begin{cases}
1 & \text{ if }\eta(n)\ll \kappa(n)
\text{ or }\eta(n)\sim\beta \kappa(n),\\
f_t\Big[\frac{\eta(r_t)}{r_t^{2/\alpha}}\Big]^{\frac{1}{\alpha-2}} & 
\text{ if }\eta(n)\gg \kappa(n)\text{ and }\alpha>2,\\
a_t\exp\left(-\frac{r_t}{\eta(r_t)f_t}\right) & \text{ if }\eta(n)\gg \kappa(n)\text{ and }\alpha=2,
\end{cases}
\end{align*}
and note that $1 \ll \theta_t \ll a_t$ in the supercritical regime by \eqref{fg2c} and since $r_t = a_t^\alpha$. Let $\mathcal{K}_t$ be the subset of non-duplicated sites with potential exceeding the threshold
\begin{align*}
\mathcal{K}_t=\big\{z\in \Z: |z|\in E\cap [1,\Za]\text{ and }\xi(z)>\theta_t\big\}.
\end{align*}
Observe that in the subcritical and critical regimes $\mathcal{K}_t$ contains \emph{all} non-duplicated sites between $-\Za$ and $\Za$, whereas in the supercritical regime $\mathcal{K}_t$ consists only of the non-duplicated sites between $-\Za$ and $\Za$ with high values of $\xi$. Denote by $\mathcal{K}_t^+$ and $\mathcal{K}_t^{-}$ the subsets of $\mathcal{K}_t$ consisting of the points lying between $0$ and $\Za$, and $0$ and $-\Za$, respectively.
\smallskip

In order to state our formalisation of \eqref{heur1}, we shall also need to guarantee certain typical properties of the subset $\mathcal{K}_t$, as well as controlling potential values near $\Za$ and $-\Za$. These are contained in the event
\begin{align*}
\mathcal{E}^2_t
&=\Big\{ f_t <\frac{\theta_t^{\alpha}|\mathcal{K}^+_t|}{\eta(r_t)}<g_t, \,  f_t <\frac{\theta_t^{\alpha}|\mathcal{K}^-_t|}{\eta(r_t)}<g_t ,  \\
& \phantom{aaaaaaa}  \big[\Za-\alpha,\Za+\alpha\big]\cap \N\subset D, \, 2\xi(z)<\xi(\Za) \text{ for all }\, 0<|z-\Za|\le \alpha \Big\}.
\end{align*}
Observe that, if $\eta(n) \to \infty$, then on the event $\mathcal{E}_t^2$,
\begin{align}
\label{fgk}
 \log g_t, \log(1/f_t) \ll \log |\mathcal{K}^+_t| ,
 \end{align}
by~\eqref{fg2a} in the subcritical and critical regimes, ~\eqref{fg2b} in the supercritical regime. To ease notation, we combine the typical properties introduced thus far into the event 
\[ \mathcal{E}_t =  \mathcal{E}^1_t \cap \mathcal{E}^2_t . \]

We are now ready to formalise \eqref{heur1}. For each $t>0$, let $\mathcal{F}_t$ be the $\sigma$-algebra generated by $D$, $\Za$, $\mathcal{K}_t$ and $\{\xi(z):z\notin \mathcal{K}_t\}$. For each $z\in\mathcal{K}_t$, let
\begin{align*}
Q_t(z)=-\log\Big(1-\frac{\xi(z)}{\xi(\Za)}\Big)\one\{\xi(z)<\xi(\Za)\}\one\{\mathcal{E}_t\},
\end{align*}
and denote 
\begin{align*}
Q_t^+=\sum_{z\in\mathcal{K}_t^+}Q_t(z)
\ , \quad 
Q_t^-= \sum_{z\in\mathcal{K}_t^-}Q_t(z)
\qquad \text{and} \qquad Q_t=Q_t^+ - Q_t^-.
\end{align*}

\begin{prop} 
\label{p:pq}
There is an $\mathcal{F}_t$-measurable random variable $P_t$ such that
\begin{align}
\label{e:pq}
\log\frac{u(t,\Za)}{u(t,-\Za)} = Q_t + P_t + o(1) ,
\end{align}
where the $o(1)$ term tends to zero almost surely on $\mathcal{E}_t$ in the non-critical regimes, and in probability in the critical regime. In the subcritical and critical regimes $P_t=0$. 
\end{prop}

Proving Proposition \ref{p:pq} is the cornerstone of the paper, and is undertaken in Sections \ref{sec:paths} and~\ref{sec:iso}. The analysis is similar to in \cite{MPS}, but considerably more involved for reasons we explain now. \smallskip

In \cite{MPS}, our overall approach was to isolate a small number of sites and show that: (i) only paths that visit these sites once make a non-negligible contribution to the solution; (ii) hence, as in \eqref{e:pq}, we may represent the ratio $u(t,-\Za)/u(t,\Za)$ as a sum over these sites; and (iii) the contribution to the ratio $u(t,-\Za)/u(t,\Za)$ from just these sites already has sufficient fluctuations to determine the behaviour of the model. There was a clear balance in choosing these sites: we needed enough for step (iii) to be available, but few enough that steps (i) and (ii) were still possible. In that paper it turned out to be sufficient in step (iii) to take an arbitrarily slowly-growing number of sites, so the equivalent statement to~\eqref{e:pq} was relatively easy to prove.
\smallskip

In the present paper we use a similar technique, but the balance is much more delicate. In particular, it is no longer sufficient to take an arbitrarily slowly-growing number of sites. In the subcritical and critical regimes we find that it is possible, and sufficient, to take all the sites in $E \cap[1,\Za]$ to make (a variant of) this argument work. In the supercritical regime, by contrast, it is not possible to capture enough of the solution from paths visiting these sites a small number of times. Instead we identify further subsets $\mathcal{K}_t^+\subset E\cap[1,\Za]$ and $\mathcal{K}_t^-\subset (-E)\cap[-\Za,1]$ such that the fluctuations of
\begin{equation}
\label{heur3}
 \xi(\Za)^{-1} \left[\sum_{z\in \mathcal{K}_t^+}\xi(z)-\sum_{z\in\mathcal{K}_t^-}\xi(z)\right].
\end{equation}
already tend to infinity; naturally, we choose these as top order statistics in order to maximise the scale of the fluctuations. It turns out that our choice of $\mathcal{K}_t^-$ and $\mathcal{K}_t^+$ are \emph{just} sufficient to guarantee enough fluctuations in~\eqref{heur3}; we explain this further when we examine Step $4$ of the proof below.\smallskip

\smallskip

\textbf{Step $3$: The typical properties.}
The next task is to establish that the typical properties contained in $\mathcal{E}_t$ hold eventually with overwhelming probability. In particular we establish the following.

\begin{prop}
\label{p:e}
$\Prob(\mathcal{E}_t)\to 1$ as $t\to\infty$.
\end{prop}

The proof of Proposition \ref{p:e} is given in Section \ref{sec:typ}, using a combination of point process techniques developed in Section \ref{sec:pp} and more direct methods; the analysis is similar to in \cite{MPS}, so we do not describe it in further details here. One immediate consequence is that, in combination with Proposition~\ref{u12}, the proof of Theorem~\ref{t:main0} is  complete.
\smallskip

\textbf{Step $4$: Fluctuation theory.}
The final step is to use standard theory to study the scale of the fluctuations of $Q_t$. To state the main result in this step, recall the $\sigma$-algebra $\mathcal{F}_t$, and let $\text{Prob}_{\mathcal{F}_t}$, ${\mathrm E}_{\mathcal{F}_t}$ and $\text{Var}_{\mathcal{F}_t}$ denote, respectively, the conditional probability, expectation and variance with respect to~$\mathcal{F}_t$. Define centred, rescaled versions of $Q_t(z)$ and $Q_t$,
\begin{align}
\label{defv}
V_{t}(z)=\frac{Q_t(z)-\mathrm{E}_{\mathcal{F}_t}Q_t(z)}
{\sqrt{\text{\rm Var}_{\mathcal{F}_t}Q_t}} \ , \qquad
V_t=\sum_{z\in\mathcal{K}_t^+}V_t(z)-\sum_{z\in\mathcal{K}_t^-}V_t(z) =  \frac{Q_t-\mathrm{E}_{\mathcal{F}_t}Q_t}
{\sqrt{\text{\rm Var}_{\mathcal{F}_t}Q_t}},
\end{align}
and denote
\begin{align*}
F_{V_t}(x)=\mathrm{E}_{\mathcal{F}_t}\one\{V_t\le x\}
\end{align*}
as the conditional distribution function of $V_t$.  \smallskip

The fluctuation theory we apply is in essence a Lindenberg central limit theorem for triangular arrays, and the consequence we draw is summarised in the following proposition. 

\begin{prop} 
\label{p:clt}
In the non-critical regimes, at $t \to \infty$, 
\begin{align}
\label{var2}
\text{\rm Var}_{\mathcal{F}_t} Q_t
\to \left\{\begin{array}{ll}
0 & \text{ if } \eta(n)\ll \kappa(n),\\
\infty &\text{ if } \eta(n)\gg \kappa(n),
\end{array}\right.
\end{align}
in probability.
In the critical regime, as $t\to\infty$,  
\begin{align}
\label{var3}
\text{\rm Var}_{\mathcal{F}_t}Q_t \Rightarrow 2\beta \sigma^2 B^2,
\end{align}
where $\sigma$ and $B$ are as in Theorem~\ref{t:main3}.
\smallskip

In the critical and supercritical regimes, as $t \to \infty$,
\begin{align}
\label{clt}
\sup_{x \in \mathbb{R}} |F_{V_t} (x) - \Phi(x) |   \to 0 \quad \text{in probability} ,
\end{align}
where $\Phi$ denotes the distribution function of a standard normal random variable.
\end{prop}

The proof of Proposition \ref{p:clt} is given in Section \ref{sec:fluct}. A key simplification in the proof is to exploit the fact that~$\Za$ is the maximiser of $\Psi_t$ over $D$ and not over $\N$ (as was the case in \cite{MPS}); this ensures that, conditionally on $\mathcal{F}_t$, the random variables $Q_t(z), z \in \mathcal{K}_t,$ are an i.i.d.\ sequence. \smallskip

Combined with Proposition \ref{p:pq}, we are now able to give a full asymptotic description of the solution in the subcritical and critical regimes, completing the proofs of Theorems \ref{t:main1} and \ref{t:main3}. In the supercritical regime, we make use of the fact that, conditional on $\mathcal{F}_t$, the random variable $Q_t$ converges in the limit to a normally distributed random variable (in fact, all that is required is that the limiting measure be non-atomic). Combined with Proposition \ref{p:pq} and the growth of $\text{\rm Var}_{\mathcal{F}_t} Q_t$, after averaging over the $\sigma$-algebra $\mathcal{F}_t$ we deduce Theorem \ref{t:main2}.\smallskip

To conclude, we present heuristics as to why our choice of $\mathcal{K}_t$ in the supercritical regime is \emph{just} sufficient to guarantee enough fluctuations in \eqref{heur3}. To get a sense of the scale of \eqref{heur3}, we appeal to the results in \cite{CM86} on the sum of top order statistics of i.i.d.\ Pareto random variables (although we cannot apply these results in our setting, they are useful for the discussion). From \cite{CM86}, the fluctuations of the sum of the top $k$ order statistics of $\eta(\Za)$ i.i.d.\ Pareto random variables is of order $\eta(\Za)^{1/\alpha} k^{1/2-1/\alpha}$. Using the typical properties of $|\mathcal{K}_t|$ in $\mathcal{E}_t^2$, equation \eqref{heur2}, and the symmetry of the model, the approximate scale of \eqref{heur3} is therefore 
\[ \xi(\Za)^{-1} \eta(\Za)^{1/\alpha} \left( \eta(\Za) \left[ \eta(\Za) / r_t^{-2/\alpha} \right]^{-\alpha/(\alpha-2)} f_t^{-\alpha} \right)^{1/2-1/\alpha} \approx f_t^{-(\alpha-2)/2}, \]
and hence tends to infinity but only barely, from where we conclude the result. \smallskip

In the case $\alpha=2$, the sum of the top $k$ order statistics is instead $\eta(\Za)^{1/2}(\log k)^{1/2}$, and so the scale of \eqref{heur3} is approximately
\[ \xi(\Za)^{-1} \eta(\Za)^{1/2} \left( \alpha f_t^{-1} r_t / \eta(r_t)   - \log \left( r_t/\eta(r_t) \right)\right)^{1/2} \approx  \alpha f_t^{-1/2}  .\]
Again this tends to infinity but only barely, and so we reach the same conclusion.

\subsection{Completing the proofs of the main results}
We finish this section by completing the proofs of Theorems \ref{t:main0}--\ref{t:main3}, assuming the four key intermediate propositions \ref{u12}--\ref{p:clt}.

\begin{proof}[Proof of Theorem~\ref{t:main0}]
This is a straightforward combination of Propositions \ref{u12} and \ref{p:e}.
\end{proof}

\begin{proof}[Proof of Theorem~\ref{t:main1}]
It suffices to show that, as $t\to\infty$, 
\begin{align*}
\log\frac{u(t,\Za)}{u(t,-\Za)}\to 0
\qquad\text{in probability}.
\end{align*}
By Proposition~\ref{p:pq} we have, for each $c>0$, eventually as $t \to \infty$,
\begin{align*}
\Prob\Big(\Big|\log\frac{u(t,\Za)}{u(t,-\Za)}\Big|>c\Big)
\le\Prob\Big(\Big\{|Q_t|>\frac c 2\Big\}\cap \mathcal{E}_t\Big)
+\Prob(\mathcal{E}_t^c).
\end{align*} 
The last term tends to zero by Proposition~\ref{p:e}. For the first term we have
\begin{align}
\label{yyy1}
\Prob\Big(\Big\{|Q_t|>\frac c 2\Big\}\cap \mathcal{E}_t\Big) \le \mathrm{E} \Big[\Prob_{\mathcal{F}_t}\Big(|Q_t|>\frac c 2\Big)\one\{\mathcal{E}_t\}\Big]+2\Prob(\mathcal{E}_t^c),
\end{align}
where the error term $\Prob(\mathcal{E}_t^c)$ arises since $\mathcal{E}_t$, which needs to be taken out of the conditional probability, is not $\mathcal{F}_t$-measurable. Taking into account that $\mathrm{E}_{\mathcal{F}_t}Q_t=0$ since
$|\mathcal{K}_t^+|=|\mathcal{K}_t^-|$ in this regime, we have by Chebychev's inequality and Proposition~\ref{p:clt}
\begin{align*}
\Prob_{\mathcal{F}_t}\Big(|Q_t|>\frac c 2\Big)
\le \frac{4}{c^2}\text{\rm Var}_{\mathcal{F}_t}Q_t\to 0
\end{align*}
almost surely on the event $\mathcal{E}_t$. Hence the expression in~\eqref{yyy1} converges to zero by the dominated convergence theorem.
\end{proof}

\begin{proof}[Proof of Theorem~\ref{t:main2}]
Similar to the proof of Theorem~\ref{t:main1}, by Proposition~\ref{p:e} it suffices to show that, for each $c>0$, as $t\to\infty$, 
\begin{align*}
\Prob\Big(\Big\{\Big|\log\frac{u(t,\Za)}{u(t,-\Za)}\Big|<c\Big\}\cap \mathcal{E}_t\Big)\to 0.
\end{align*}
By Proposition~\ref{p:pq} it is then enough to prove that, as  $t\to\infty$, 
\begin{align*}
\Prob\big(\big\{|Q_t + P_t|<2c\big\}\cap \mathcal{E}_t\big)\to 0,
\end{align*}
for which, in turn, it suffices to show that
\begin{align*}
\mathrm{E} \Big[\Prob_{\mathcal{F}_t}\big(|Q_t + P_t|<2c\big)\one\{ \mathcal{E}_t\}\Big]\to 0.
\end{align*}
Observe that, similarly to the proof of Theorem~\ref{t:main1}, the event $\mathcal{E}_t$ can be taken out of the conditional probability since its probability tends to one.  Now, by the dominated convergence theorem, it remains to prove that
\begin{align}
\label{yyy2}
\Prob_{\mathcal{F}_t}\big(|Q_t  + P_t|<2c\big)\to 0
\end{align}
almost surely on $\mathcal{E}_t$. Observe that 
\begin{align}\label{eq:qtvt}
Q_t=V_t\sqrt{\text{\rm Var}_{\mathcal{F}_t}Q_t} +\mathrm{E}_{\mathcal{F}_t}Q_t.
\end{align}
Hence~\eqref{yyy2} is equivalent to showing that, almost surely
\begin{align}
\label{yyy3}
\Prob_{\mathcal{F}_t}\Big(V_t\in \big[\text{\rm Var}_{\mathcal{F}_t}Q_t\big]^{-\frac 1 2}  (-P_t-\mathrm{E}_{\mathcal{F}_t}Q_t-2c, - P_t- \mathrm{E} _{\mathcal{F}_t}Q_t+2c)\Big)\to 0.
\end{align}
Since $P_t$, $\mathrm{E}_{\mathcal{F}_t}Q_t$, and $\text{\rm Var}_{\mathcal{F}_t}Q_t$ are $\mathcal{F}_t$-measurable, and the length of the interval on the right-hand side of $\in$ tends to zero by~\eqref{var2} in Proposition~\ref{p:clt}, \eqref{yyy3} now follows from~\eqref{clt} there.
\end{proof}

\begin{proof}[Proof of Theorem~\ref{t:main3}] By Proposition~\ref{p:pq} 
and using~\eqref{eq:qtvt}, there is an event $\mathcal{E}_t^\ast$ with probability tending to one such that
\begin{align*}
\log\frac{u(t,\Za)}{u(t,-\Za)}=V_t\sqrt{\text{Var}_{\mathcal{F}_t}Q_t}+o(1)
\end{align*}
almost surely on $\mathcal{E}_t^\ast$, since $\mathrm{E}_{\mathcal{F}_t}Q_t=0$ due to the fact that $|\mathcal{K}_t^+|=|\mathcal{K}_t^-|$ in this regime. Hence, first restricting on the event $\mathcal{E}_t\cap \mathcal{E}_t^\ast$ and then dropping the restriction on $\mathcal{E}_t^\ast$ as it is no longer needed, we have by Proposition~\ref{p:e} for every $x\in \R$
\begin{align*}
\Prob\Big(\log\frac{u(t,\Za)}{u(t,-\Za)}<x\Big)
&-\Prob\Big(\big\{V_t\sqrt{\text{Var}_{\mathcal{F}_t}Q_t}+o(1)<x\big\}\cap \mathcal{E}_t \Big)\\
&\le \Prob(\mathcal{E}_t^c)+2\Prob((\mathcal{E}_t^\ast)^c)\to 0.
\end{align*}
Further, again by Proposition~\ref{p:e} we have
\begin{align*}
\Prob\Big(\big\{V_t\sqrt{\text{Var}_{\mathcal{F}_t}Q_t}+o(1)<x\big\}\cap \mathcal{E}_t\Big)
- \mathrm{E} \Big[F_{V_t}\Big(\frac{x+o(1)}{\sqrt{\text{Var}_{\mathcal{F}_t}Q_t}}\Big)\one\{\mathcal{E}_t\}\Big]
\le 2\Prob(\mathcal{E}_t^c)\to 0,
\end{align*}
where the error term arises since $\mathcal{E}_t$ is not $\mathcal{F}_t$-measurable. By~\eqref{clt} in Proposition~\ref{p:clt} and the dominated convergence theorem we have 
\begin{align*}
\mathrm{E} \Big[F_{V_t}\Big(\frac{x+o(1)}{\sqrt{\text{Var}_{\mathcal{F}_t}Q_t}}\Big)\one\{\mathcal{E}_t\}\Big]
- \mathrm{E} \Big[\Phi\Big(\frac{x+o(1)}{\sqrt{\text{Var}_{\mathcal{F}_t}Q_t}}\Big)\one\{\mathcal{E}_t\}\Big]\to 0.
\end{align*}
Now by uniform continuity of $\Phi$ and Proposition~\ref{p:e} we obtain 
\begin{align*}
 \mathrm{E} \Big[\Phi\Big(\frac{x+o(1)}{\sqrt{\text{Var}_{\mathcal{F}_t}Q_t}}\Big)\one\{\mathcal{E}_t\}\Big]
- \mathrm{E}\Big[\Phi\Big(\frac{x}{\sqrt{\text{Var}_{\mathcal{F}_t}Q_t}}\Big)\Big]
\to 0.
\end{align*}
Finally, since $\Phi$ is continuous and bounded we have by~\eqref{var3} in Proposition~\ref{p:clt}  
\begin{align*}
\mathrm{E} \Big[\Phi\Big(\frac{x}{\sqrt{\text{Var}_{\mathcal{F}_t}Q_t}}\Big)\Big]
\to \mathrm{E}_* \Big[\Phi\Big(\frac{x}{\sqrt{2\beta} \sigma B}\Big)\Big]
\end{align*}
as required, where $\mathrm{E}_*$ denotes expectation with respect to $B$.
\end{proof}


\smallskip
\section{Preliminaries}
\label{sec:prelim}

In this section we state some preliminary results. We begin by establishing Proposition \ref{u12}, which follows closely the proof of the equivalent statement in \cite{MPS}. Next, we develop point process machinery that allows us to control the asymptotic behaviour of the high points of $\xi$; this machinery will be heavily used in Sections \ref{sec:iso}--\ref{sec:fluct}. Finally, we derive asymptotic properties of the functions $\eta$ and $N$, and also of the subset $\mathcal{K}_t$.

\subsection{Proof of Proposition \ref{u12}}

The first step is to establish the negligibility of paths that, up to time $t$, either (i) make more than $R_t$ jumps or, (ii) do not hit either of the sites $\Za$ or $-\Za$, i.e.\ that, as $t \to \infty$, almost surely
\begin{align}
\label{e:u1}
 U(t)^{-1}  \E \Big[ \exp\Big\{\int_0^t \xi(X_s)ds\Big\} \one\{J_t > R_t \ \text{ or } \  \tau_{\{-\Za, \Za\} } > t \} \Big]  \mathbf{1}_{\mathcal{E}^1_t} \to 0 , 
 \end{align}
where $\tau_A=\inf\{t>0:X_t\in A\}$ is the hitting time of the set $A$ by $(X_s)$. Equation \eqref{e:u1} follows in a near identical manner to the proof of~\cite[Prop.\ 3.7]{MPS}, in particular making use of the assumption~\eqref{fg1} on the decay and growth of $f_t$ and $g_t$. The only difference is that now $\Za$ and $\Zb$ have been defined as maximisers of $\Psi_t$ over $D$ rather than $\Z$. This, however, is taken care of by the condition $\Psi_t(\Zepm)<\Psi_t(\Zb)$ in $\mathcal{E}^1_t$.
\smallskip

To conclude the proof, we establish the negligibility of the contribution to $U(t)$ from paths satisfying
\[   J_t \le R_t \ , \quad  \tau_{\{-\Za, \Za\} } \le t  \quad \text{and} \quad   X_t \notin  \{-\Za, \Za\}    , \]
i.e.\ that, as $t \to \infty$, almost surely
\begin{align*}
 U(t)^{-1}  \E \Big[ \exp\Big\{\int_0^t \xi(X_s)ds\Big\} \one\{   J_t \le R_t , \,   \tau_{\{-\Za, \Za\} } \le t  , \,   X_t \notin  \{-\Za, \Za\}    \} \Big]  \mathbf{1}_{\mathcal{E}^1_t} \to 0 , 
\end{align*}
This follows in an identical manner to the proof of~\cite[Lem.\ 4.4]{MPS}, in particular using the fact that, on the event $\mathcal{E}^1_t$, $ \xi(\Za) - \xi(z) > a_t f_t$ for all $z \in [-R_t, R_t] \setminus \{-\Za, \Za\}$.

\subsection{Point processes machinery}
\label{sec:pp}

In this section we develop point process machinery that allows us to analyse the high points of the potential field. Our approach is similar to in \cite{MPS} and elsewhere (see, e.g., \cite{KLMS, ST}), but contains several new ideas necessary to handle the critical regime. \smallskip

The main result of the section establishes the convergence of (a rescaled version of) the potential field to a Poisson point process. In the critical regime, we simultaneously establish the convergence of just the non-duplicated potential values. This latter convergence is rather non-standard: in order to get a non-trivial limit we need to use a different scaling to the standard one (used in \cite{KLMS, MPS, ST} for instance), and the limiting Poisson point process is spatially inhomogeneous. \smallskip

We begin by describing the Poisson point processes that appear in the limit. Abbreviate $\rho=1/(\alpha-1)$. To describe the limit of the potential field, we work in the state space
\begin{align*}
G = \{(x,y): x\ge 0, y>\rho x\}  
\end{align*}
equipped with a topology in which a set is relatively compact if and only if its distance from the line $y=\rho x$ is positive, and let $\Pi$ be a Poisson point process with the intensity measure 
\[\mu(dx\otimes dy)=dx\otimes \frac{\alpha}{|y|^{\alpha+1}}dy \]
on the state space $G$. Observe that $\mu$ is finite for every relatively compact set in $G$ according to~\cite[Eq.\ (5.2)]{MPS}, which ensures that $\Pi$ is well-defined. To describe the limit of the non-duplicated potential in the critical case, we instead work in the state space
\[ \hat{G} =  [0,\infty) \times (0,\infty] \]
equipped with its usual topology, and let $\hat \Pi$ be a Poisson point process with the intensity measure 
\[ \hat \mu(dx\otimes dy)=\frac{2\beta}{\alpha}x^{\frac{2}{\alpha}-1}dx\otimes \frac{\alpha}{|y|^{\alpha+1}}dy \]
on the state space $\hat{G}$, independent of $\Pi$. Again observe that $\hat\mu$ is finite for every relatively compact set in $\hat{G}$, which ensures that $\hat \Pi$ is well-defined. Finally, let $\varnothing$ be the empty point process on $G$. In the sequel, we denote by the same symbols the restriction of these point processes to subsets of $[0,\infty)\times\R$. Denote the probability and expectation corresponding to the above Poisson point processes (as well as all other Poisson point processes defined in this section) by $\Prob_{*}$ and $\mathrm{E}_*$. 
\smallskip

We next define the rescaled versions of the potential that will converge to the limit processes defined above. As mentioned, we need different scaling in order to examine the potential and the non-duplicated potential respectively, and for the latter the scaling further depends on whether $\alpha > 2$ or $\alpha = 2$. As such, define the following scaled versions of the potential
\begin{align*}
\Pi^{\ssup{d}}_s&=\sum_{z\in D}\e_G\Big(\frac{z}{s}, \frac{\xi(z)}{s^{1/\alpha}}\Big),
\qquad
\Pi^{\ssup e}_s=\sum_{z\in E}\e_G\Big(\frac{z}{s}, \frac{\xi(z)}{s^{1/\alpha}}\Big),\\
\hat\Pi^{\ssup e}_s&=\sum_{z\in E}\e\Big(\frac{z}{s}, \frac{\xi(z)}{s^{2/\alpha^2}}\Big),\qquad\tilde\Pi^{\ssup e}_s=\sum_{z\in E}\e\Big(\frac{z}{s}, \frac{\xi(z)}{(s/\log s)^{1/2}}\Big),
\end{align*}
where $\e(x,y)$ denotes the Dirac measure placing mass on $(x,y)$ and $\e_G(x,y)$ denotes its restriction to~$G$. It is easy to check that these measures are all almost surely finite on $G$ since $\rho>0$ and $\alpha>1$. Denote $\Pi_s=\Pi^{\ssup d}_s+\Pi^{\ssup e}_s$. \smallskip

The main proposition in this section is the following.

\begin{prop}
\label{lppp}
As $s \to \infty$, $(\Pi^{\ssup {d}}_s, \Pi^{\ssup e}_s)$ converges in law to $(\Pi, \emptypp)$. In particular $\Pi_s$ converges in law to $\Pi$. In the critical regime with $\alpha>2$, $(\Pi^{\ssup {d}}_s, \hat\Pi^{\ssup e}_s)$ converges in law to $(\Pi, \hat\Pi)$, whereas in the critical regime with $\alpha=2$, $(\Pi^{\ssup {d}}_s, \tilde\Pi^{\ssup e}_s)$ converges in law to $(\Pi, \hat\Pi)$, both as $s\to\infty$.
\end{prop}

Before proving Proposition \ref{lppp}, we state and prove a simple lemma deriving the asymptotic behaviour of $q$ from that of $\eta$. The proof is a straightforward consequence of the fact that $q(n) \to 0$ is eventually decreasing.

\begin{lemma} 
\label{1101}
In the critical regime 
\[ q(n)\sim \begin{cases}\frac{2\beta}{\alpha} n^{\frac{2}{\alpha}-1}&\mbox{if }\alpha>2,\\
\frac{\beta}{\log n}&\mbox{if }\alpha=2,
\end{cases} \] 
as $n\to\infty$.   
\end{lemma}

\begin{proof} 
We focus on the $\alpha>2$ case; the $\alpha=2$ is handled in a similar way. Denote $a=2/\alpha$ for brevity. Let $\e>0$ and choose $\delta>0$ so small that the graphs $y=(1+x)^a$ and $y=\frac{1+\delta}{1-\delta}+\frac{1-\e}{1-\delta}ax$ intersect in two 
positive points, which we denote $x_1<x_2$.  Since $\eta(n)\sim\beta n^a$ there exists $n_0\in \N$
such that $(1-\delta)\beta n^a<\eta(n)<(1+\delta)\beta n^a$ for all $n\ge n_0$, and also $q(n)$ is decreasing for all $n\ge n_0$. 
Let us show that $q(n)>(1-\e)\beta a n^{a-1}$ eventually. 
\smallskip

Suppose it is not the case. Then we can find $n\ge n_0$ and $k\in \N$ such that $q(n)\le (1-\e)\beta a n^{a-1}$ and $k/n \in (x_1,x_2)$. By monotonicity we have  
\begin{align*}
(1-\delta) \beta(n+k)^a<\eta(n+k)=\eta(n)+\sum_{i=n+1}^{n+k}q(i)\le (1+\delta)\beta n^a+(1-\e)\beta k a  n^{a-1} .
\end{align*}
Hence 
\begin{align*}
\Big(1+\frac k n\Big)^a<\frac{1+\delta}{1-\delta}+\frac{1-\e}{1-\delta}\cdot a\cdot \frac k n, 
\end{align*}
which contradicts $ k/ n \in (x_1,x_2)$. Hence $\liminf\limits_{n\to\infty} q(n) / (\beta a n^{a-1}) \ge 1$, and $\limsup\limits_{n\to\infty} q(n) / (\beta a n^{a-1})\le 1$ is similar.
\end{proof}

\begin{proof}[Proof of Proposition \ref{lppp}]
The proof follows the lines of~\cite[Lem.\ 5.1]{MPS} and starts in the same way for each statement. 
Define the point process
\begin{align*}
\Sigma_s=\sum_{z\in \N_0}\e_G\Big(\frac{z}{s},\frac{\xi(z)}{s^{1/\alpha}}\one\{z\in D\}\Big)
+\sum_{z\in \N}\e_{\bar G}\Big(\frac{z}{s},-\frac{\xi(z)}{\omega_s}\one\{z\in E\}\Big),
\end{align*}
where $\bar G$ is the reflection with respect to the $x$-axis of $G$ (for the first statement) or 
$\hat{G}$ (for the second and third statements), $\e_{\bar G}$ is the restriction of the Dirac measure on $\bar G$, 
and $\omega_s=s^{1/\alpha}$ (for the first statement), $\omega_s=s^{2/\alpha^2}$  (for the second statement) or $\omega_s=(s/\log s)^{1/2}$ (for the third statement). Let $\Sigma$ be a Poisson point process on $G\cup \bar G$ with the intensity measure $\mu^{*}$ which equals $\mu$ on $G$ and is zero (for the first statement) or $\hat\mu$ (for the second and third statements) on $\bar G$.\smallskip

It suffices to show that $\Sigma_s$ converges in law to $\Sigma$ on the state space $G\cup \bar G$, 
as $\Pi_s^{\ssup{d}}$ can be represented 
by the restriction of $\Sigma_s$ to the upper half plane 
and $\Pi_s^{\ssup e}$ (for the first statement), $\hat \Pi_s^{\ssup e}$ (for the second statement) or $\tilde\Pi_s^{(e)}$ (for the third statement)  by the restriction of $\Sigma_s$ to the lower half plane reflected with respect to the $x$-axes. 
\smallskip

Let $\mathcal{C}_K^+$ denote the set of positive continuous functions $h:G\cup \bar G\to\mathbb{R}$ with compact support. For any $s$, denote by
\begin{align*}
\mathcal{L}_s(h)=\mathrm{E} \exp\Big\{-\int hd\Sigma_s\Big\}
\qquad\text{and}\qquad
\mathcal{L}(h) =  \mathrm{E}_* \exp\Big\{-\int hd\Sigma\Big\},
\end{align*}
the Laplace transforms of $\Sigma_s$ and $\Sigma$, where $h\in \mathcal{C}_K^+$. We will denote by the same symbol the extension of $h$ by zero to $\R^2$. 
Recall from~\cite[Prop.\ 3.6]{resnick} that since $\Sigma$ is a Poisson point process its Laplace transform is given by 
\begin{align}
\label{lt}
\log \mathcal{L}(h)=-\iint_{G\cup\bar G}(1-e^{-h(x,y)})\mu^*(dx,dy).
\end{align}
By~\cite[Prop.\ 3.19]{resnick} it suffices to show that $\mathcal{L}_t(h) \to \mathcal{L}(h)$ for all $h\in \mathcal{C}_K^+$.
\smallskip

Suppose the support of $h$ is contained in  $\{(x,y):x\ge 0, y\ge \rho x+c\}\cup \{(x,y):x\ge 0, y\le -\rho x-c\}$
(for the first statement), in 
$\{(x,y):x\ge 0, y\ge \rho x+c\}\cup \big( [0,\infty)\times [-\infty,-c]\big)$ (for the second statement), or in $\{(x,y):x\ge 0, y\ge \rho x+c\}\cup \big( [0,c)\times [-\infty,-c]\big)$ (for the third statement)
for some $c>0$. 
Following the lines of~\cite[Lem.\ 5.1]{MPS}, we have, 
using the convention $p(0)=1$ and $q(0) = 0$,  
\begin{align}
\log \mathcal{L}_s(h)
=\sum_{z=0}^{\infty}\log \Big(
&p(z)\mathrm{E}\Big[\exp\Big\{-h
\Big(\frac{z}{s},\frac{\xi(z)}{s^{1/\alpha}}\Big)
\one\Big\{\Big(\frac{z}{s},\frac{\xi(z)}{s^{1/\alpha}}\Big)\in G
\Big\}\Big\}\Big]\notag\\
+&q(z)\mathrm{E}\Big[\exp\Big\{-h
\Big(\frac{z}{s},-\frac{\xi(z)}{\omega_s}\Big)
\one\Big\{\Big(\frac{z}{s},-\frac{\xi(z)}{\omega_s}\Big)\in \bar G\Big\}\Big\}
\Big]\Big).
\label{ee5}
\end{align}

Integrating with respect to $\xi$ and rescaling the variable by $s^{1/\alpha}$, we compute 
\begin{align}
\mathrm{E}&\Big[\exp\Big\{-h
\Big(\frac{z}{s},\frac{\xi(z)}{s^{1/\alpha}}\Big)
\one\Big\{\Big(\frac{z}{s},\frac{\xi(z)}{s^{1/\alpha}}\Big)\in G
\Big\}\Big\}\Big]\notag\\
&=1-\frac{1}{s}\int_{0}^{\infty} \Big[1-\exp\Big\{-h
\Big(\frac{z}{s},u\Big)\one\Big\{\Big(\frac{z}{s},u\Big)\in G\Big\}\Big\}
\Big]\frac{\alpha du}{u^{\alpha+1}}
\label{ee1}
\end{align}
for all $s$ such that $s^{-1/\alpha}<c$, and the integral is uniformly bounded since 
\begin{align}
\label{iub}
\int_{0}^{\infty} \Big[1-\exp\Big\{-h
\Big(\frac{z}{s},u\Big)\one\Big\{\Big(\frac{z}{s},u\Big)\in G\Big\}\Big\}\Big\}
\Big]
\frac{\alpha du}{u^{\alpha+1}}
\le \int_{c}^{\infty} 
\frac{\alpha du}{u^{\alpha+1}}<\infty.
\end{align}

For the first statement, the second expectation in~\eqref{ee5} is treated similarly to the first one. 
For the second and third statements, we integrate with respect to $\xi$ and rescale the variable by~$\omega_s$
\begin{align}
\mathrm{E}&\Big[\exp\Big\{-h
\Big(\frac{z}{s},-\frac{\xi(z)}{\omega_s}\Big)
\one\Big\{\Big(\frac{z}{s},-\frac{\xi(z)}{\omega_s}\Big)\in \bar G
\Big\}\Big\}\Big]\notag\\
&=1-\frac{1}{\omega_s^{\alpha}}\int_{0}^{\infty} \Big[1-\exp\Big\{-h
\Big(\frac{z}{s},-u\Big)\one\Big\{\Big(\frac{z}{s},-u\Big)\in \bar G\Big\}\Big\}
\Big]
\frac{\alpha du}{u^{\alpha+1}}
\label{ee11}
\end{align}
for all $s$ such that $1/\omega_s<c$, and the integral is uniformly bounded by~\eqref{iub}.
\smallskip

Doing the Taylor expansion we obtain 
\begin{align}
\log \mathcal{L}_s(h)=
&-\sum_{z=0}^{\infty}\frac{p(z)}{s}\int_{0}^{\infty} \Big[1-\exp\Big\{-h
\Big(\frac{z}{s},u\Big)\one\Big\{\Big(\frac{z}{s},u\Big)\in G\Big\}\Big\}\Big]
\frac{\alpha du}{u^{\alpha+1}}\notag\\
&-\sum_{z=0}^{\infty}\frac{q(z)}{\omega_s^{\alpha}}\int_{0}^{\infty} \Big[1-\exp\Big\{-h
\Big(\frac{z}{s},-u\Big)\one\Big\{\Big(\frac{z}{s},-u\Big)\in \bar G\Big\}\Big\}\Big]
\frac{\alpha du}{u^{\alpha+1}}+o(1).
\label{ee12}
\end{align}

For the first statement we use  $p(z)\to 1$, $q(z)\to 0$ and $\omega_s^{\alpha}=s$ in order to conclude that
the second term in~\eqref{ee12} disappears and 
$\log \mathcal{L}_s(h) $ converges to 
\begin{align*}
-\int_0^{\infty}\int_0^{\infty}\big(1-\exp\big\{-h(x,y)\one\{(x,y)\in G\}\big\}\big)\frac{\alpha dxdy}{y^{\alpha+1}}=\log \mathcal{L}(h)
\end{align*}
given by~\eqref{lt}.
\smallskip

For the second statement, we have by Lemma~\ref{1101} ($\alpha>2$ case) 
\begin{align*}
\frac{q(z)}{\omega_s^{\alpha}}=\frac{2\beta}{\alpha s}\Big(\frac{z}{s}\Big)^{\frac{2}{\alpha}-1}(1+o(1)),
\end{align*}
where $o(1)$ is with respect to $z\to\infty$ and independent of $s$. Using this for the second term in~\eqref{ee12}
and $p(z)\to 1$ for the first term, we also arrive at $\log \mathcal{L}_s(h) \to \log \mathcal{L}(h) $ given by \eqref{lt}.
\smallskip

For the third statement, we have by Lemma~\ref{1101} ($\alpha=2$ case)
\[ \frac{q(z)}{\omega_s^2}=\frac{\beta}{s}\cdot\frac{\log s}{\log z}(1+o(1)), \]
where $o(1)$ is with respect to $z\to\infty$ and independent of $s$. We write the second term in~\eqref{ee12} as
\begin{align*}
&\sum_{z=0}^{\left\lfloor \frac{s}{(\log s)^2}\right\rfloor}\frac{q(z)}{\omega_s^2}\int_0^\infty\Big[1-\exp\Big\{-h
\Big(\frac{z}{s},-u\Big)\one\Big\{\Big(\frac{z}{s},-u\Big)\in \bar G\Big\}\Big\}\Big]
\frac{2 du}{u^{3}}\\
+&\sum_{z=\left\lfloor \frac{s}{(\log s)^2}\right\rfloor+1}^{\infty}\frac{q(z)}{\omega_s^2}\int_0^\infty\Big[1-\exp\Big\{-h
\Big(\frac{z}{s},-u\Big)\one\Big\{\Big(\frac{z}{s},-u\Big)\in \bar G\Big\}\Big\}\Big]
\frac{2 du}{u^{3}}, 
\end{align*}
and we show that the first term here is negligible. Indeed we can upper bound its absolute value by 
\[ \frac{\log s}{s} \left(\left\lfloor \frac{s}{(\log s)^2}\right\rfloor+1\right)\int_0^\infty\frac{2 du}{u^3}\to0 \]
as $s\to\infty$. Now observe that for $z\ge s/(\log s)^2$ and $z\le cs$ we have $\log s/\log z \sim 1$ as $s\to\infty$, and so we again arrive at $\log\mathcal{L}_s(h)\to\log \mathcal{L}(h)$ given by~\eqref{lt}.
\end{proof}

We next establish the convergence of certain functionals of the above point processes. This allows us to state a scaling limit for the maximiser $\Za$ and its potential value $\xi(\Za)$, as well as to give asymptotic properties of other high values of the potential. \smallskip

Given a point measure $\Sigma$, we say that $x\in\Sigma$
if $\Sigma(\{x\})>0$.  
Let the positive random variables 
$X^{\ssup 1}, X^{\ssup 2}$ and $Y^{\ssup 1}, Y^{\ssup 2}$ be defined by the properties that 
\begin{align*}
(X^{\ssup 1}, Y^{\ssup 1}) &\in \Pi, \text{ and if } (x,y) \in\Pi\text{ then }y-\rho x\le Y^{\ssup 1} -\rho X^{\ssup 1}, \\
(X^{\ssup 2}, Y^{\ssup 2}) &\in \Pi,
\text{ and if } (x,y) \in \Pi  \setminus\{(X^{\ssup 1}, Y^{\ssup 1})\} \text{ then }y-\rho|x|\le Y^{\ssup 2}-\rho X^{\ssup 2} .
\end{align*}
It can be proved in the same way as in~\cite[Lem.\ 5.2]{MPS}
that, almost surely, the random variables $X^{\ssup 1}, X^{\ssup 2}, Y^{\ssup 1}$ and $Y^{\ssup 2}$ are well-defined and satisfy $Y^{\ssup 1}-\rho X^{\ssup 1} >Y^{\ssup 2}-\rho X^{\ssup 2} >0$. 
\smallskip

Denote by $Z_t^{\ssup e}$ a maximiser of $\Psi_t$ over $E$, 
and by $Z_t^{\ssup{1*}}$ and $Z_t^{\ssup{2*}}$ the first and second maximisers of $\Psi_t$ over $\N_0$, respectively. Their existence is standard. 

\begin{prop}\label{LB1}
As $t\to\infty$, 
\begin{align}
\label{oo1}
\Big(\frac{\Za}{r_t},\frac{\Zb}{r_t},\frac{\xi(\Za)}{a_t},\frac{\xi(\Zb)}{a_t}\Big)
&\Rightarrow (X^{\ssup 1},X^{\ssup 2},Y^{\ssup 1},Y^{\ssup 2}),\\
\Big(\frac{\Psi_t(\Za)}{a_t},\frac{\Psi_t(\Za)}{a_t}\Big)
& \Rightarrow (Y^{\ssup 1}-\rho X^{\ssup 1}, Y^{\ssup 2}-\rho X^{\ssup 2}).
\label{oo2}
\end{align}
In particular, 
\begin{align}
\label{lB}
\frac{(\Za)^{1/\alpha}}{\xi(\Za)}\Rightarrow \frac{(X^{\ssup 1})^{1/\alpha}}{Y^{\ssup 1}},
\end{align}
and the density of the pair $(X^{\ssup 1}, Y^{\ssup 1})$ is given by 
\begin{align*}
p(x,y)= \alpha y^{-\alpha-1}\exp\big\{-(y-\rho x)^{1-\alpha}\big\}\one\{y > \rho x > 0\}.
\end{align*}
Further,
\begin{align}
\label{oo7}
\Prob\big(Z_t^{\ssup e}\notin\{ Z_t^{\ssup{1*}} , Z_t^{\ssup{2*}} \} \big) =  \Prob\big(\Za=Z_t^{\ssup{1*} }, \Zb=Z_t^{\ssup{2*}}\big)\to 1.
\end{align}
\end{prop}

\begin{proof} 
First, we claim that the convergences~\eqref{oo1} and~\eqref{oo2} hold if we replace $\Za$ and $\Zb$ by $Z_t^{\ssup{1*}}$ and $Z_t^{\ssup{2*}}$. The proof of this relies on $\Pi_s \Rightarrow \Pi$, which follows from Proposition~\ref{lppp}, but is otherwise the same as that of~\cite[Prop.\ 5.5]{MPS}. 
\smallskip

Second, observe that 
\begin{align*}
\frac{\Psi_t(z)}{a_t}=\frac{\xi(z)}{a_t}-\rho \frac{z}{r_t} +o(1)\frac{z}{r_t}+
o(1)\frac{z}{r_t}\log\frac{\xi(z)}{a_t}
\end{align*}
and hence, for any $c_1,c_2>0$, 
we have, using $a_t^{\alpha}=r_t$, 
\begin{align*}
\Prob&\Big(\frac{\Psi_t(Z_t^{\ssup e})}{a_t}>c_1, \frac{Z_t^{\ssup e}}{r_t}<c_2,c_1<\frac{\xi(Z_t^{\ssup e})}{a_t}<c_2\Big)
\le \Prob\Big(\frac{\xi(Z_t^{\ssup e})}{a_t}-\rho \frac{Z_t^{\ssup e}}{r_t}>c_1/2\Big)\\
&\le \Prob\Big(\Pi_{r_t}^{\ssup e}\big(\big\{(x,y):x\ge 0, y>\rho x+c_1/2\big\}\big)\neq 0\Big)\\
&\to \Prob_{*} \Big(\emptypp\big(\big\{(x,y):x\ge 0, y>\rho x+c_1/2\big\}\big)\neq 0\Big)=0
\end{align*}
by Proposition~\ref{lppp}. Since $c_1$ can be chosen arbitrarily small and $c_2$ can be chosen arbitrarily large,  
we obtain, taking into account the first step of the proof, that 
\begin{align*} 
\Prob\big(\Psi_t(Z_t^{\ssup e})<\Psi_t(Z_t^{\ssup{2*}})  \big)\to 1.
\end{align*}
We have now established~\eqref{oo7}, and hence~\eqref{oo1} and~\eqref{oo2} follow from the first step of the proof. 
The weak convergence ~\eqref{lB} is an obvious consequence of~\eqref{oo1}, and the density $p$ can be computed similarly to~\cite[Lem.\ 5.3]{MPS}.
\end{proof}

\begin{lemma} 
\label{l:os}
As $t\to\infty$, 
\begin{align*}
\Prob\big(\xi(\Za)-\xi(z)>a_tf_t\text{ \rm for all } |z|  \in [0, R_t] \setminus \Za  \big)\to 1.
\end{align*} 
\end{lemma}

\begin{proof} It was shown in Proposition~\ref{LB1} that
$\Za=Z_t^{\ssup{1*}}$ 
with overwhelming probability. Hence, and by symmetry (notice that the potential values of non-duplicated sites on the negative half-line and positive half-line are independent and equal in law), it suffices to show that
\begin{align*}
\Prob\big(\xi(Z_t^{\ssup{1*}})-\xi(z)>a_tf_t\text{ for all }  z \in [0, R_t] \setminus \Za \big)\to 1.
\end{align*} 
This can be shown in the same way as in~\cite[Prop. 5.6]{MPS}. The proof uses $\Pi_s\Rightarrow \Pi$, which follows from Proposition~\ref{lppp}.
\end{proof}

\begin{lemma} 
\label{l:zeta}
Denote 
\begin{align*}
\zeta_t=\max_{|z|\in E, |z|< \Za }\xi(z).
\end{align*}
In the critical regime, as $t\to\infty$,
\begin{align*}
\lambda(t)^{1/2}\frac{\zeta_t}{a_t^{2/\alpha}}\Rightarrow F,
\end{align*}
where $F$ is a strictly positive random variable. 
\end{lemma}

\begin{proof}
Observe that, for any $c>0$,  
\begin{align*}
\Prob\Big(\lambda(t)^{1/2}\frac{\zeta_t}{a_t^{2/\alpha}}<c\Big)=
\Big[\Prob\Big(\lambda(t)^{1/2}\frac{\zeta^+_t}{a_t^{2/\alpha}}<c\Big)\Big]^2,
\end{align*}
where 
\begin{align*}
\zeta^+_t=\max_{z\in E, z<\Za}\xi(z).
\end{align*}
It can be shown in the same way as in~\cite[Lem.\ A2]{MPS} that 
\begin{align}
\label{psibar}
\Prob\Big(\Za\text{ is the maximiser of }z\mapsto \xi(z)-\frac{\rho z}{t}\log t\text{ over }D\Big)\to 1. 
\end{align}
The proof only requires $\Pi_s^{\ssup d}\Rightarrow \Pi$, which we have by Proposition~\ref{lppp}. In other words, 
with overwhelming probability $\Za$ is defined by the property that 
\begin{align}
\label{iii}
\frac{\xi(z)}{a_t}-\rho\, \frac{z}{r_t}\le \frac{\xi(\Za)}{a_t}-\rho\, \frac{\Za}{r_t}
\text{ for all }z\in D.
\end{align}
Denote $G_{\delta}=\{(x,y):x\ge 0, y\ge \rho x +\delta\}$. By Proposition~\ref{lppp} for the critical case, using $r_t=a_t^{\alpha}$, Proposition~\ref{LB1} and~\eqref{iii} we have, for $\alpha>2$,
\begin{align*}
& \! \! \Prob\Big(\frac{\zeta^+_t}{a_t^{2/\alpha}}<c, \Big( \frac{\Za}{r_t},\frac{\xi(\Za)}{a_t}\Big)\in G_{\delta}\Big)\\
&=\int_{G_{\delta}}\Prob\Big(
\hat \Pi_{r_t}^{\ssup e}\big([0,x]\times [c,\infty)\big)=0, \Pi_{r_t}^{\ssup d}(dx\times dy)=1, \Pi_{r_t}^{\ssup d}
\big(\{(u,v):u-\rho v>x-\rho y\}\big)=0\Big)\\
&\to\int_{G_{\delta}} \Prob_{*} \Big(
\hat \Pi\big([0,x]\times [c,\infty)\big)=0\Big) \Prob_{*} \Big( \Pi(dx\times dy)=1, \Pi
\big(\{(u,v):u-\rho v>x-\rho y\}\big)=0\Big)\\
&=\int_{G_{\delta}}\exp\Big\{-\hat\mu\big([0,x]\times [c,\infty)\big)\Big\}p(x,y)dxdy
=\int_{G_{\delta}}\exp\big\{-\beta x^{2/\alpha}c^{-\alpha}\big\}p(x,y)dxdy.
\end{align*}
Since by Proposition~\ref{LB1}
\begin{align*}
\lim_{\delta\downarrow 0}\Prob\Big( \Big( \frac{\Za}{r_t},\frac{\xi(\Za)}{a_t}\Big)\not\in G_{\delta}\Big)= 0
\end{align*}
this implies 
\begin{align*}
\Prob\Big(\frac{\zeta^+_t}{a_t^{2/\alpha}}<c\Big)
= \mathrm{E}_* \exp\big\{-\beta (X^{\ssup 1})^{2/\alpha}c^{-\alpha}\big\}
\end{align*}
where on the right-hand side we have a distribution function of a positive random variable. The argument for $\alpha=2$ is the same except with $\hat\Pi_{r_t}^{(e)}$ replaced with $\tilde\Pi_{r_t}^{(e)}$.
\end{proof}

\subsection{The asymptotic behaviour of non-duplicated sites}

In this section we derive asymptotic statements for $N,\eta$ and $|\mathcal{K}_t|$. We begin with a general result about convergence of random variables which we will use in the proof of Lemma~\ref{l:kk}; we omit the proof since it is standard.

\begin{lemma}\label{l:conv}
Suppose $(X_t)$ is a random process with values in $(0,\infty)$ such that $X_t\to0$ in probability.
Suppose $h_t$ are complex-valued functions on $(0,\infty)$ bounded by one, and $a\in \mathbb{C}$ such that
\[
\lim_{\eta\to0}\lim_{t\to\infty}\sup_{x\in(0,\eta)}|h_t(x)- a| = 0.\]
Then $\mathbb{E}h_t(X_t) \to a$ as $t \to\infty$.
\end{lemma}

We now give a basic lemma on the asymptotic behaviour of $Z_t$.
\begin{lemma}
\label{l:0}
As $t \to \infty$, $\Za \to \infty$ almost surely. Moreover, $\Za/r_t$ is bounded away from zero and infinity in probability.
\end{lemma}
\begin{proof}
The first statement is standard, see Lemma~3.2 of~\cite{KLMS}; the second statement follows from Proposition~\ref{LB1}.
\end{proof}

We next give asymptotic properties of $N, \eta$ and $\mathcal{K}_t$, which allow us to derive the order of $|\mathcal{K}_t|$.
\begin{lemma} 
\label{l:nn}
If $\eta(n)$ converges then $N(n)$ converges almost surely. If $\eta(n)\to \infty$ then, as $n\to\infty$, 
\begin{align*}
\frac{N(n)}{\eta(n)}\to 1 \qquad \text{in probability}. 
\end{align*}
\end{lemma}

\begin{proof} 
If $\eta(n)\to\eta<\infty$ then by Chebychev's inequality we have 
\begin{align*}
\text{Prob}(N(n)>c)\le \frac{\eta(n)}{c}\le \frac{\eta}{c}\to 0
\end{align*}
as $c\to\infty$. Hence $P(N(n) \to\infty)=0$ and $N(n)$ converges almost surely since it is increasing.  
\smallskip

If instead $\eta(n)\to \infty$ we have 
\begin{align*}
\mathrm{E}\exp\Big\{it\frac{N(n)}{\eta(n)}\Big\}
&=\prod_{z=1}^n \Big[1-q(z)+q(z)\exp\Big\{\frac{it}{\eta(n)}\Big\}\Big]
=\exp\Big\{\sum_{z=1}^{n}\log \Big(1+\frac{itq(z)}{\eta(n)}(1+o(1))\Big)\Big\}\\
&=\exp\Big\{\sum_{z=1}^{n}\frac{itq(z)}{\eta(n)}(1+o(1))\Big\}=\exp{it+o(1)}\to \exp\{it\}
\end{align*}
as required. 
\end{proof}
Recall from the beginning of Section~\ref{subsecphase} that we have extended the function $\eta$ so that $\eta:\mathbb{R}_+\to \mathbb{R}$.
\begin{lemma} 
\label{l:etaeta}
As $t\to\infty$, $\eta(\Za)/\eta(r_t)$ is bounded away from zero and infinity in probability.
\end{lemma}

\begin{proof}
Let $n\in\N$ be such that $q(i)$ is decreasing for all $i\ge n$. For $0<a<1$ and $x>n/a$ we have 
\begin{align}
\label{ax}
\eta(ax)\ge \int_{n}^{ax}\!\! q(u)du=a\int_{n/a}^x \!\! q(av)dv\ge a\int_{n/a}^x \!\! q(v)dv=a\eta(x)-a\eta(n/a).
\end{align}
Similarly, for $a>1$ and $x>n$ we have 
\begin{align*}
\eta(ax)=\eta(an)+ \int_{an}^{ax}\!\! q(u)du
=\eta(an)+a\int_{n}^x \!\! q(av)dv
\le \eta(an)+a\int_{n}^x \!\! q(v)dv
\le a\eta(x)+\eta(an). 
\end{align*} 
This implies, for all $x>n\max\{1,1/a\}$, 
\begin{align*}
\left.\begin{array}{r} 
\displaystyle a\Big(1-\frac{\eta(n/a)}{\eta(x)}\Big)\\
\displaystyle 1
\end{array}
\right\}
\le\frac{\eta(ax)}{\eta(x)}\le
\left\{\begin{array}{ll} 1 & \qquad\text{ if }0<a<1,\\
\displaystyle
a+\frac{\eta(n)}{\eta(x)}& \qquad\text{ if }a\ge 1.\end{array}\right.
\end{align*}
It remains to apply these inequalities to $x=r_t$ and $a=\Za/r_t$, and the statement of the lemma follows from Lemma~\ref{l:0}.
\end{proof}

\begin{lemma}
\label{l:kk}
If $\eta(n)$ converges then $\mathcal{K}_t$ is bounded almost surely. If $\eta(n)\to \infty$ then, as $t\to\infty$,  
\begin{align}
\label{bip0}
\frac{\theta_t^{\alpha}|\mathcal{K}_t^+|}{N(\Za)}\to 1 
\qquad\text{and}\qquad
\frac{\theta_t^{\alpha}|\mathcal{K}_t^-|}{N(\Za)}\to 1 
\qquad
\text{in probability}. 
\end{align}
\end{lemma}

\begin{proof} 
If $\eta(n)$ converges then $E$ is bounded almost surely by Lemma~\ref{l:nn} and hence so is $\mathcal{K}_t^+$. Suppose instead $\eta(n)\to\infty$. Denote by $\mathcal{G}$ the $\sigma$-algebra generated by $D$ and $\{\xi(z):z\in D\}$, and denote the conditional expectation with respect to $\mathcal{G}$ by $\mathbb{E}_\mathcal{G}$. It is easy to see that, conditionally on $\mathcal{G}$, the events $\big\{z\in \mathcal{K}_t^+\big\}_{z \in E,z<\Za} $ are independent and have the same probability $\theta_t^{-\alpha} \le 1$. Hence $|\mathcal{K}_t^+|$ is a binomial random variable with parameters $N(\Za)$ and $\theta_t^{-\alpha}$. Thus the characteristic function of $\frac{\theta_t^{\alpha}|\mathcal{K}_t^+|}{N(\Za)}$ satisfies, for each $\lambda\in\mathbb{R}$,
\begin{align*}
\mathbb{E}\left[i\lambda\frac{\theta_t^{\alpha}|\mathcal{K}_t^+|}{N(\Za)}\right]=\mathbb{E}\left[\mathbb{E}_\mathcal{G}\left[i\lambda\frac{\theta_t^{\alpha}|\mathcal{K}_t^+|}{N(\Za)}\right]\right]&=\mathbb{E}\left[\left(1-\theta_t^{-\alpha}+\theta_t^{-\alpha}\exp\left\{\frac{i\lambda\theta_t^{\alpha}}{N(\Za)}\right\}\right)^{N(\Za)}\right]\\
&=\mathbb{E}h_{\theta_t^{-\alpha}}\left(\frac{\theta_t^\alpha}{N(\Za)}\right),
\end{align*}
where \[h_\delta(x)=\left(1-\delta+\delta e^{i\lambda x}\right)^{1/\delta x}.\]We will apply Lemma~\ref{l:conv} to show that this converges to $e^{i\lambda}$, which will complete the proof. In order to do so we must show that 
\begin{align}\label{eq:thetaN}\frac{\theta_t^{\alpha}}{N(\Za)}\to0\end{align}
 in probability, and 
 \begin{align}\label{eq:limsuph}
 \lim_{\eta\to0}\lim_{\delta\to0}\sup_{x\in(0,\eta)}|h_\delta(x)-e^{i\lambda}|=0.
 \end{align}

\smallskip
 
In the subcritical and critical regimes \eqref{eq:thetaN} follows from $\theta_t=1$ and Lemmas~\ref{l:0} and~\ref{l:nn}, which imply that 
that $N(\Za)\to\infty$ in probability. In the supercritical regime we have 
\begin{align*}
\frac{N(\Za)}{\theta_t^{\alpha}}=\frac{N(\Za)}{\eta(\Za)}
\cdot\frac{\eta(\Za)}{\eta(r_t)}\cdot \begin{cases}
 f_t^{-\alpha}\cdot \Big[\frac{r_t}{\eta(r_t)}\Big]^{\frac{2}{\alpha-2}}&\to \infty\quad\mbox{if }\alpha>2,\\
 \exp\left(\frac{2r_t}{\eta(r_t)f_t}-\log\frac{r_t}{\eta(r_t)}\right)&\to\infty\quad\mbox{if }\alpha=2,
 \end{cases}
\end{align*}
in probability by Lemmas~\ref{l:nn} and \ref{l:etaeta} and since $\eta(r_t)\ll r_t$, and so again \eqref{eq:thetaN} holds.

\smallskip
For \eqref{eq:limsuph}, observe that by doing a Taylor expansion with the remainder in Lagrange's form, we have
\[
h_\delta(x)=\exp\left\{\frac1{\delta x}\log\left(1-\delta+\delta e^{i\lambda x}\right)\right\}=\exp\left\{i\lambda +\frac{r_\delta(x)}{\delta x}\right\},
\]
where it can be shown that $|r_\delta(x)|\le 2\delta x^2\lambda^2$ for all $x$ sufficiently small. Hence, doing another Taylor expansion, we obtain
\[
|h_\delta(x)-e^{i\lambda}|=\Big|\exp\left\{\frac{r_\delta(x)}{\delta x}\right\}-1\Big|\le 2\Big|\frac{r_\delta(x)}{\delta x}\Big|\sup_{\kappa\in(0,x)}|e^\kappa|\le 5x\lambda^2,
\]
for all $x$ small enough, which clearly implies \eqref{eq:limsuph}.
\end{proof}

\begin{corollary}
\label{c:kk}
As $t\to\infty$, 
\begin{align}
\label{bip}
\frac{\theta_t^{\alpha}|\mathcal{K}_t^+|}{\eta(r_t)}
\qquad\text{ and }\qquad
\frac{\theta_t^{\alpha}|\mathcal{K}_t^-|}{\eta(r_t)}\end{align}
are bounded away from zero and infinity in probability. 
\end{corollary}
\begin{proof}
This is a straightforward combination of Lemmas~\ref{l:nn}-\ref{l:kk}.
\end{proof}


\bigskip
\section{Significant paths}
\label{sec:paths}

We now embark on the proof of Proposition \ref{p:pq}. The first step, carried out in this section, is to further eliminate from consideration a class of paths that make a non-negligible contribution to the total solution $U(t)$. The extra control we gain over the remaining paths will be crucial in allowing us to represent, in Section \ref{sec:iso}, the ratio $u(t, \Za)/u(t, -\Za)$ as a sum over non-duplicated sites. \smallskip

Of course, by Theorem~\ref{t:main0} it is already enough to consider only paths such that $X_t \in \{-\Za,\Za\}$. Such a path necessarily makes at least $ |\mathcal{K}_t^+|$ or $|\mathcal{K}_t^-|$ visits to the set $\mathcal{K}_t$ by time $t$, depending on the endpoint of the path. Here we further eliminate paths that either, up to time $t$: (i) visit the set $\mathcal{K}_t$ too many additional times beyond the minimum; (ii) return to the set $\{-\Za,\Za\}$ too frequently; or (iii) make too long a loop originating from the set $\{-\Za,\Za\}$. \smallskip

In the critical regime we find that, unlike in \cite{MPS}, it is not possible to consider only paths that \textit{never} make additional visits to sites in $\mathcal{K}_t$; this would not give us the dominant portion of the solution. Instead, we need to consider paths that make a small number of extra visits, and argue later than this makes no significant difference to our representation of the ratio $u(t, \Za)/u(t,-\Za)$ as a sum over non-duplicated sites. \smallskip

We begin by introducing some path notation, which mirrors the set-up in \cite{MPS}. Denote by 
\begin{align*}
\mathcal{P}_{all}=\{y=(y_0,\dots,y_{\ell})\in \Z^{\ell+1}:\ell\in\N_0,|y_i-y_{i-1}|=1\text{ for all }1\le i\le \ell\}
\end{align*}
the set of all geometric paths on $\Z$.  For each path $y\in\mathcal{P}_{all}$, denote by $\ell(y)$ its length (counted as the number of edges). Denote by $(\tau_i)_{i\in \N_0}$
the sequence of the jump times of the continuous-time random walk $(X_t)$ and by 
\begin{align*}
P(t,y)=\{X_0=y_0, \, X_{\tau_0+\cdots+\tau_{i-1}}=y_i\text{ for all }1\le i\le \ell(y),  \, t-\tau_{\ell(y)}\le\tau_0+\cdots+\tau_{\ell(y)-1}<t\} 
\end{align*}
the event that the random walk has the trajectory $y$ up to time $t$. Let
\begin{align*}
U(t,y)=\E\Big[\exp\Big\{\int_0^t\xi(X_s)ds\Big\}\one_{P(t,y)}\Big]
\end{align*}
be the contribution of the event $P(t,y)$ to $U(t)$. 
\smallskip

Let $\mathcal{P}^t$ denote the subset of paths in $\mathcal{P}_\text{all}$ that start at the origin, end in $\{-\Za,\Za\}$, and have length at most $R_t$. For any $y\in \mathcal{P}^t$, the \emph{skeleton} \smallskip of $y$, denoted $\text{skel}(y)$, is the geometric path from the origin to a site in $\{-\Za,\Za\}$ constructed by chronologically removing all loops in $y$ which start and end at any site belonging to $\{0\}\cup \mathcal{K}_t$ up until the first visit of $\{-\Za,\Za\}$, as well as removing any part of the path after the final visit of $y$ to $\{-\Za,\Za\}$. 
\smallskip

We can now partition $\mathcal{P}^t$ into equivalence classes by saying that paths $y$ and $\hat y$ are in the same class if and only if $\text{skel}(y)=\text{skel}(\hat y)$. We write $\mathfrak{P}^t$ for the set of all such equivalence classes.
Note that any such equivalence class 
$\mathcal{P}\in\mathfrak{P}^t$ contains the \emph{null path}, $y_{\mathrm{null}}^\mathcal{P}\in\mathcal{P}^t$, defined as $y_{\mathrm{null}}^\mathcal{P}=\mathrm{skel}(y_{\mathrm{null}}^\mathcal{P})$. Observe that every null path, prior to visiting $\{-\Za,\Za\}$ for the first time, either (i) visits each site in $\{0\} \cup \mathcal{K}_t^+$ exactly once, or (ii) visits each site in $\{0\} \cup \mathcal{K}_t^-$ exactly once. In particular, until the first visit of $\{-\Za,\Za\}$ each null path visits either only positive integers, or only negative integers. 
\smallskip

Denote by $\mathcal{P}^t_{1}$ the subset of $\mathcal{P}^t$ consisting of all paths $y$ having at most $\lfloor w_t\rfloor$ extra visits to the set  
$\mathcal{K}_t$ compared to $\text{skel}(y)$, where 
\begin{align*}
w_t=\left\{\begin{array}{ll}
g_t^3/f_t^2 & \text{ in the critical regime,}\\
0 & \text{ otherwise.}
\end{array}\right.
\end{align*}
Let $\Lambda_t= 9et /\xi(\Za)$, and denote by $\mathcal{P}^t_{2}$ the subset of $\mathcal{P}^t$ consisting of all paths $y$ having at most $\Lambda_t$ returns to the set $\{-\Za,\Za\}$, and making no loops from the set $\{-\Za,\Za\}$ of length more than $2\alpha$.\smallskip

The main result of this section is the following pair of lemmas, which together establish the negligibility of paths not in $ \mathcal{P}^t_{0} =\mathcal{P}^t_{1}\cap \mathcal{P}^t_{2}$. 

\begin{lemma}
\label{L:nullpaths}
Almost surely, as $t \to \infty$,
\begin{align*}
U_0(t)=(1+o(1))\sum_{y\in \mathcal{P}^t_{1}}U(t,y)
\end{align*}
on the event $\mathcal{E}_t$.
\end{lemma}  

\begin{lemma} 
\label{l:null0}
Almost surely, as $t \to \infty$,
\[ U_0(t)=(1+o(1))\sum_{y\in \mathcal{P}^t_{2}}U(t,y) \]
on the event $\mathcal{E}_t$.
\end{lemma}
 
\begin{proof}[Proof of Lemma \ref{L:nullpaths}] 
Given an equivalence class $\mathcal{P}\in\mathfrak{P}^t$,
we write 
$\mathcal{P}_w$ for the subset of $\mathcal{P}$ consisting of the paths with additional length $2w$ before the first visit to $\{-\Za,\Za\}$, compared to $y_{\mathrm{null}}^\mathcal{P}$. It suffices to show that 
\begin{align*}
\sum_{w> w_t}\sum_{ y\in\mathcal{P}_{w}}U(t,y) 
&=o(1) U(t,y_{\mathrm{null}}^\mathcal{P})
\end{align*}
uniformly for all $\mathcal{P}\in\mathfrak{P}^t$ on the event $\mathcal{E}_t$, as $t\to\infty$.
\smallskip

For each $w\in \mathbb{N}_0$ we have  
\begin{align*}
|\mathcal{P}_{w}|\le 2^{2w} {w+|\mathcal{K}_t|\choose w} ,
\end{align*}
since half of the additional $2w$ pieces (all pointing in one direction) are chosen according to the number of weak compositions of $w$ into at most $|\mathcal{K}_t|+1$ parts, while the remaining half (all pointing in the other direction) have to be assigned to the sites in $\{0\}\cup\mathcal{K}_t$  in a unique way to form loops. \smallskip
 
For each $y\in\mathcal{P}_{w}$, we have on $\mathcal{ E}_t$, similarly to~\cite[Lem.\ 4.4]{MPS},  
\begin{align*}
U(t,y)\le U(t,y_{\mathrm{null}}^\mathcal{P})\prod_{j=1}^{2w}
\frac{1}{\xi(\Za)-c_{j}}<U(t,y_{\mathrm{null}}^\mathcal{P})(a_tf_t)^{-2w},
\end{align*}
where $c_1,\dots,c_{2w}$ are the values of $\xi$ in the additional points of $y$ compared to $y_{\mathrm{null}}^\mathcal{P}$. Hence 
\begin{align}
\label{qq3}
\sum_{w> w_t}\sum_{ y\in\mathcal{P}_{w}}U(t,y) 
&< U(t,y_{\mathrm{null}}^\mathcal{P})\sum_{w> w_t} 
{w+|\mathcal{K}_t|\choose w}\Big(\frac{2}{a_tf_t}\Big)^{2w}
\end{align}
on $\mathcal{E}_t$. Observe that ${n \choose k}\le 2^n$ for all $0\le k\le n$, which implies
\begin{align}
\label{qq4}
\sum_{w\ge |\mathcal{K}_t|} 
{w+|\mathcal{K}_t|\choose w}\Big(\frac{2}{a_tf_t}\Big)^{2w}\le 
\sum_{w\ge |\mathcal{K}_t|} 
2^{w+|\mathcal{K}_t|}\Big(\frac{2}{a_tf_t}\Big)^{2w}
\sim \Big(\frac{4}{a_tf_t}\Big)^{2|\mathcal{K}_t|}\to 0.
\end{align}
Further, on $\mathcal{E}_t$
\begin{align}
\label{qq5}
\sum_{w_t< w<|\mathcal{K}_t|} 
{w+|\mathcal{K}_t|\choose w}\Big(\frac{2}{a_tf_t}\Big)^{2w}\le \sum_{w_t< w<|\mathcal{K}_t|} 
\frac{1}{w!}\Big(\frac{8|\mathcal{K}_t|}{a^2_tf^2_t}\Big)^{w}
< \sum_{w> w_t} 
\frac{1}{w!}\Big(\frac{8g_t\eta(r_t)}{\theta_t^{\alpha}a^2_tf^2_t}\Big)^{w}.
\end{align}
It remains to prove that \eqref{qq5} converges to $0$, since then \eqref{qq3}-\eqref{qq5} yields the result. \smallskip

We first analyse \eqref{qq5} in the non-critical regimes. Observe that by~\eqref{fg2c} and since $\eta(r_t) \ll r_t$ and $r_t=a_t^{\alpha}$ we have
\begin{align*}
\frac{g_t\eta(r_t)}{\theta_t^{\alpha}a_t^2f_t^2}
=\begin{cases}
\frac{g_t}{f_t^2}\cdot
\frac{\eta(r_t)}{r_t^{2/\alpha}}
&\to 0\quad
\text{ if }\eta(n)\ll \kappa(n),\\
\frac{g_t}{f_t^{2+\alpha}}
\Big[\frac{\eta(r_t)}{r_t^{2/\alpha}}\Big]^{-\frac{2}{\alpha-2}}&\to 0\quad
\text{ if }\eta(n)\gg \kappa(n)\text{ and }\alpha>2,\\
\frac{g_t}{f_t^2}\cdot\frac{\eta(r_t)}{r_t}\exp\left(-\log r_t+\frac{2r_t}{\eta(r_t)f_t}\right)&\to0\quad\text{if }\eta(n)\gg\kappa(n)\text{ and }\alpha=2,
\end{cases}
\end{align*}
which implies that 
\begin{align*}
\sum_{w=1}^{\infty} 
{w+|\mathcal{K}_t|\choose w}\Big(\frac{2}{a_tf_t}\Big)^{2w}\to 0.
\end{align*}
\smallskip 

We now consider the critical regime. Since $\eta(n)\sim \beta \kappa(n)$, eventually 
\begin{align*}
\frac{8g_t\eta(r_t)}{\theta_t^{\alpha}a_t^2f_t^2} = \frac{8g_t\eta(r_t)}{f_t^2 r_t^{2/\alpha}} < \frac{g_t^2}{f_t^2}.
\end{align*}
We use concentration of Poisson random variables to analyse \eqref{qq5}. Denote by $W_t$ a Poisson random variable with mean $g_t^2/f_t^{2}$, and let $\mathbf{P}$ and $\mathbf{E}$ denote its probability and expectation. Then for any $\theta>0$ by the exponential Chebychev inequality
\begin{align*}
\sum_{w\ge w_t} 
\frac{1}{w!}\Big(\frac{8g_t\eta(r_t)}{\theta_t^{\alpha}a^2_tf^2_t}\Big)^{w}
& <  \sum_{w\ge w_t}
\frac{1}{w!}\Big(\frac{g_t^2}{f_t^2}\Big)^{w}
=\exp\Big\{\frac{g_t^2}{f_t^2}\Big\} \mathbf{P}(W_t\ge w_t)\\
&\le \exp\Big\{\frac{g_t^2}{f_t^2}-\theta w_t\Big\}
\mathbf{E} e^{\theta W_t}
= \exp\Big\{\frac{g_t^2}{f_t^2}e^{\theta}-\theta w_t\Big\} .
\end{align*}  
Optimising over $\theta$, we use the minimiser 
\[ \theta=\log\frac{w_tf_t^2}{g_t^2}=\log g_t \]
 and get 
\begin{align*}
\sum_{w\ge w_t} 
\frac{1}{w!}\Big(\frac{8g_t\eta(r_t)}{\theta_t^{\alpha}a^2_tf^2_t}\Big)^{w}
&\le \exp\Big\{\frac{g_t^3}{f_t^2}-\frac{g_t^3}{f_t^2}\log g_t\Big\}\to 0
\end{align*} 
as required.
\end{proof}

\begin{proof}[Proof of Lemma \ref{l:null0}] 
It can be proved in the same way as~\cite[Lem.\ 4.5]{MPS} that paths which contain a loop from the set $\{-\Za,\Za\}$ that is longer than $2\alpha$ make a negligible contribution to $U_0(t)$. Let us show that the contributions from paths with more than $\Lambda_t$ returns to the set $\{-\Za,\Za\}$ is also negligible. 
\smallskip

We may assume that the length of any such loop is no more than $2\alpha$. We shall split paths from~$\mathcal{P}^t$ into equivalence classes by saying that two paths are equivalent if and only if they are the same after removing all loops from $\{-\Za,\Za\}$. For every path $y\in \mathcal{P}^t$ visiting $\{-\Za,\Za\}$ exactly once, write $\mathcal{P}^y_{\ell,m}$ for the class of paths equivalent to $y$ with additional length $\ell$ compared to $y$ and which visit $\{-\Za,\Za\}$ $m$ additional times.
\smallskip

By~\cite[Lem.\ 3.9]{MPS}, for any path $\hat y\in \mathcal{P}^y_{\ell,m}$ we have  on $\mathcal{E}_t$
\begin{align*}
U(t,\hat y)\le U(t,y)\frac{t^m}{m!}\prod_{j=1}^{\ell-m}\frac{1}{\xi(\Za)-c_j} < U(t,y)\frac{t^m}{m!}\Big(\frac{2}{\xi(\Za)}\Big)^{\ell-m},
\end{align*}
where $c_1,\dots,c_{\ell-m}$ are the extra values of $\hat y$ taken outside the set $\{-\Za,\Za\}$.
Since the length of the loops does not exceed $2\alpha$, we have used $c_i<\xi(\Za)/2$ for all $i$ on $\mathcal{E}_t$. Using the bound $|\mathcal{P}^y_{\ell,m}|\le 2^\ell$, we have 
\begin{align*}
\sum_{m\ge \Lambda_t}\sum_{\ell\ge 2m}
\sum_{\hat y\in \mathcal{P}^y_{\ell,m}}
U(t,\hat y)
& < U(t,y) \sum_{m\ge \Lambda_t}\sum_{\ell\ge 2m}
2^\ell\frac{t^m}{m!}\Big(\frac{2}{\xi(\Za)}\Big)^{\ell-m}
\sim U(t,y) \sum_{m\ge \Lambda_t}\frac{1}{m!}\Big(\frac{8t}{\xi(\Za)}\Big)^{m},
\end{align*}
and it suffices to show that the sum on the right-hand side tends to zero. As in the proof of Lemma~\ref{L:nullpaths}, we use concentration of Poisson random variables. Let $W_t$ be a Poisson random variable with mean $8t/\xi(\Za)$ otherwise independent of the $\sigma$-algebra generated by $D$ and $\xi$, and let $\mathbf{P}$ and $\mathbf{E}$ denote its probability and expectation. Then for any $\theta>0$ by the exponential Chebychev inequality
\begin{align*}
\sum_{m\ge \Lambda_t}\frac{1}{m!}\Big(\frac{8t}{\xi(\Za)}\Big)^{m}
&=\exp\Big\{\frac{8t}{\xi(\Za)}\Big\} \mathbf{P}(W_t\ge \Lambda_t)\\
&\le \exp\Big\{\frac{8t}{\xi(\Za)}-\theta\Lambda_t\Big\}\mathbf{E} e^{\theta W_t}
=\exp\Big\{\frac{8t}{\xi(\Za)}e^{\theta}-\theta\Lambda_t
\Big\}.
\end{align*}
Optimising over $\theta$, we use the minimiser 
\[ \theta=\log\frac{\Lambda_t \xi(\Za)}{8t}
=\log \frac{9e}{8}>1\]
 and get 
\begin{align*}
\sum_{m\ge \Lambda_t}\frac{1}{m!}\Big(\frac{4t}{f_ta_t}\Big)^{m}\le \exp \Big\{\frac{9et}{\xi(\Za)}\Big[1-\log\frac{9e}{8}\Big]\Big\}
\le  \exp \Big\{\frac{9et}{a_tg_t}\Big[1-\log\frac{9e}{8}\Big]\Big\} \to 0
\end{align*}
since $t / (a_tg_t) \to \infty$ for $\alpha\ge 2$.
\end{proof}


\bigskip
\section{Representing the ratio as a sum over non-duplicated sites}
\label{sec:iso}

In this section we represent the ratio $u(t, \Za)/u(t, \Zb)$ as a sum over the non-duplicated sites $\mathcal{K}_t$, and in particular complete the proof of Proposition \ref{p:pq}. The main preliminary result we need is the following. 

\begin{prop} 
\label{p:ext}
For every $y\in\mathcal{P}^t_{0}$, there exists an $\mathcal{F}_t$-measurable random variable 
$\Xi_t(y)$ such that the random variable $\Upsilon_t(y)$ defined by 
\begin{align}
\label{p:pq2}
U(t,y)
&=\Xi_t(y)\cdot \Upsilon_t(y)\cdot
\!\!\prod_{i:y_i\in\mathcal{K}_t}\Big(1-\frac{\xi(y_i)}{\xi(\Za)}\Big)^{-1}
\end{align}
converges, as $t \to \infty$, to one almost surely on $\mathcal{E}_t$ in the non-critical regimes and in probability in the critical regime. Moreover, this convergence holds uniformly in $y\in\mathcal{P}^t_{0}$.
\smallskip

Furthermore, for all $y\in\mathcal{P}^t_{0}$,
\begin{align*}
\Xi_t(y)=\Xi_t(-y)
\end{align*} 
in the subcritical and critical regimes.
\end{prop}

The proof of Proposition \ref{p:ext} is rather technical, especially in the critical regime, and takes up the majority of the section. The proof is subject to three auxiliary lemmas, whose statements follow the main proof. In the critical regime, we shall also need to assume fine control over the second-order contributions to the product in \eqref{p:pq2}; we introduce the relevant event now.
\smallskip

 For each $t>0$, define first and second moment functions
\begin{align}
\label{e:m}
M_t^+=\sum_{z\in \mathcal{K}_t^+}\frac{1}{\xi(\Za)-\xi(z)}
\qquad\text{and}\qquad
M_t^-=\sum_{z\in \mathcal{K}_t^-}\frac{1}{\xi(\Za)-\xi(z)} ,
\end{align}
and
\begin{align}
\label{e:sig}
(\Sigma^+_t)^2=\sum_{z\in \mathcal{K}_t^+}\frac{1}{(\xi(\Za)-\xi(z))^2}
\qquad\text{and}\qquad
(\Sigma^-_t)^2=\sum_{z\in \mathcal{K}_t^-}\frac{1}{(\xi(\Za)-\xi(z))^2}.
\end{align}
Recalling the counting function for non-duplicated sites 
\[ N(n)=\sum_{z=1}^{n}\one\{z\in E\} ,\]
 define also their approximations 
\begin{align*}
\bar{M}_t=\frac{N(\Za)}{\xi(\Za)}\Big(1+ \frac{\gamma}{\xi(\Za)}\Big)
\qquad\text{and}\qquad
\frac{1}{\bar{S}_t}=\frac{\xi(\Za)}{N(\Za)^{1/2}}\Big(1
-\frac{\gamma}{\xi(\Za)}\Big),
\end{align*}
where $\gamma = \mathrm{E} \xi(0)= \alpha /(\alpha-1)$. 
\smallskip
Let $\lambda:\mathbb{R}_+\to\mathbb{R}$ by defined as
\[ \lambda(t)=\begin{cases}1&\mbox{if }\alpha>2,\\\log t&\mbox{if }\alpha=2. \end{cases} \]
The control over second-order contributions that we need in the critical regime is summarised by 
\begin{align*}
\mathcal{E}_t^{cr}
= \Big\{& \frac{\lambda(r_t)}{a_t} \cdot \bar{M}_t<g_t,  
\, a_t\big|M_t^+ - \bar{M}_t\big|<g_t, \, a_t\big|M_t^- - \bar{M}_t\big|<g_t,\\
& \phantom{aaaa} f_t< \frac{1}{\sqrt{\lambda(r_t)} \bar{S}_t}<g_t,  \, a_t^2\Big|\frac{1}{\Sigma_t^+}-\frac{1}{\bar{S}_t}\Big|
<g_t, \, a_t^2\Big|\frac{1}{\Sigma_t^-}-\frac{1}{\bar{S}_t}\Big|
<g_t\Big\}.
\end{align*} 

This event holds eventually with overwhelming probability:
\begin{prop}
\label{p:e2}
$\Prob(\mathcal{E}_t)\to 1$ as $t\to\infty$.
\end{prop}

We defer the proof of Proposition \ref{p:e2} to Section \ref{sec:typ}, where we also prove that the typical properties contained in $\mathcal{E}_t$ hold eventually with overwhelming probability.

\begin{proof}[Proof of Proposition \ref{p:ext}]
Write $\mathcal{P}^t_{+}$, $\mathcal{P}^t_{-}$ for the set of paths in $\mathcal{P}^t_{0}$ ending in $\Za$ and $-\Za$, respectively. We shall concentrate on defining $\Xi_t(y)$ for $y \in \mathcal{P}^t_{+}$; this will depend on $\xi$ only through the values of $\xi$ at sites not in $\mathcal{K}_t^+$ and the quantity~$|\mathcal{K}_t^+|$. For $y\in \mathcal{P}^t_-$ we define $\Xi(y)$ analogously, just replacing $|\mathcal{K}_t^+|$ by $|\mathcal{K}_t^-|$. Since $|\mathcal{K}_t^+|=|\mathcal{K}_t^-|$ in the subcritical and critical regimes (both quantities are the number of non-duplicate sites between $-\Za$ and $\Za$), this establishes $\Xi_t(y)=\Xi_t(-y)$ in the subcritical and critical regimes.
\smallskip

We begin by giving a useful representation for $U(t, y)$. For $t>0$ and $n\in \N$, denote by 
\begin{align*}
S^n_t=\Big\{(x_0,\dots,x_{n-1})\in \R_+^n: 
\sum_{i=0}^{n-1}x_i<t\Big\}
\end{align*}
the $n$-dimensional simplex of size $t$. First recall that, for any path $y \in \mathcal{P}_\text{all}$, by direct computation (see~\cite[Eq.\ (2.2)]{MPS}) we have 
\begin{align}
U(t,y)
&=e^{-2t} I_{\ell(y)}(t; \xi(y_0),\dots,\xi(y_{\ell(y)})) ,
\label{uy1}
\end{align}  
where the function $I$ is defined by
\begin{align*}
I_n(t; c_0,\dots,c_n)=e^{tc_n}\int_{S^n_t}\exp\Big\{ \sum_{i=0}^{n-1}x_i (c_i-c_n)\Big\}  dx_0\cdots dx_{n-1},
\end{align*}
for each $t>0$, $n\in\N$, and $c_0,\dots,c_n\in\R$. In particular, $I_0(t; c_0)=e^{tc_0}$.
\smallskip

Let $y\in \mathcal{P}^t_{+}$ and denote by $k(y)$ the number of visits of $y$ to $\mathcal{K}_t^+$ and by $m(y)$ the number of visits of $y$ to $\Za$ minus one. Let $n(y)=\ell(y)-k(y)-m(y)$. Further denote 
\begin{align*}
\mathfrak{I}^t_0=\big\{0\le i\le \ell(y): y_i\in\mathcal{K}_t^+\big\} \, , \quad \mathfrak{I}^t_1=\big\{0\le i\le \ell(y): y_i\not\in\mathcal{K}_t^+\big\}
\end{align*}
and
\[  \mathfrak{I}^t_2=\big\{0\le i\le \ell(y): y_i\not\in(\mathcal{K}_t^+\cup \{\Za\})\big\}. \]
Using~\eqref{uy1} and rescaling by $\xi(\Za)-\xi(y_i)$ the $k(y)$ variables $x_i\in \mathfrak{I}^t_0$, which is possible on $\mathcal{E}_t$, we obtain
\begin{align*}
U(t,y)
=e^{t\xi(\Za)-2t}\prod_{i\in\mathfrak{I}^t_0}\frac{1}{\xi(\Za)-\xi(y_i)}
&\int\limits_{S_t^{\ell(y)-k(y)}}
\exp\Big\{\sum_{i\in\mathfrak{I}^t_1}x_i\big(\xi(y_i)-\xi(\Za)\big)\Big\}\notag\\
&\times J_{k(y)} \Big(t-\sum_{i\in\mathfrak{I}^t_1} x_i;\,  \xi({\bf y}\cap \mathfrak{I}^t_0), \xi(\Za)\Big)
d({\bf x}\cap \mathfrak{I}^t_1),
\end{align*}
where, for each $s>0$, 
\begin{align*}
J_{k(y)}\big(s; \xi({\bf y}\cap \mathfrak{I}^t_0), \xi(\Za)\big)
&=\int_{\R^k_+}\exp\Big\{-\sum_{i\in\mathfrak{I}^t_0} x_i\Big\}\one\Big\{\sum_{i\in\mathfrak{I}^t_0}\frac{x_i}{\xi(\Za)-\xi(y_i)}<s\Big\}
d({\bf x}\cap  \mathfrak{I}^t_0), 
\end{align*}
and $\xi({\bf y}\cap \mathfrak{I}^t_0)$ is the vector of values $\xi(y_i), i\in\mathfrak{I}^t_0$, 
$d({\bf x}\cap  \mathfrak{I})$ is an abbreviation for the product of $dx_i$ for all $i\in\mathfrak{I}$, for any index set $\mathfrak{I}$. 
\smallskip

Let $\tau_z$, $z\in\N$, and $\hat\tau_i$, $i\in\N$ be independent exponentially distributed random variables with parameter one, also independent from $\xi$ and $D$, with $\mathbb{P}$ their probability law. 
Denote by $w(y)$ the number of extra visits of $y$ to the set $\mathcal{K}_t^+$  
and denote by $z_i(y)\in \mathcal{K}_t^+$ the point of the $i^{\rm{th}}$ extra revisit.
Recall  that $w(y)\le w_t$ since $y\in \mathcal{P}^t_{+}$.
We have, for each $s\in [0,t]$,
\begin{align*}
J_{k(y)}\big(s; \xi({\bf y}\cap \mathfrak{I}^t_0), \xi(\Za)\big)
&=\mathbb{P} \Big(\sum_{z\in \mathcal{K}_t^+}\frac{\tau_z}{\xi(\Za)-\xi(z)}+W_t(y)<s\Big)\notag\\
&=\mathbb{P} \Big(\frac{1}{\Sigma_t^+}\sum_{z\in \mathcal{K}_t^+}\frac{\tau_z-1}{\xi(\Za)-\xi(z)}<\frac{1}{\Sigma_t^+}
\big[s-W_t(y)-M_t^+\big]\Big)\notag\\
&= \Phi_t\Big(\frac{1}{\Sigma_t^+}
\big[s-W_t(y)-M_t^+\big]\Big),
\end{align*}
where $M_t^+$ and $\Sigma_t^+$ are as at \eqref{e:m} and \eqref{e:sig}, 
\begin{align*}
W_t(y)=\sum_{i=1}^{w(y)}\frac{\hat\tau_i}{\xi(\Za)-\xi(z_i(y))}
\end{align*}
and $\Phi_t$ denotes the distribution function (with respect to $(\tau_z)$ only, conditionally on $\xi$ and $D$) of the random variable 
\begin{align*}
\Gamma_t=\frac{1}{\Sigma_t^+}\sum_{z\in \mathcal{K}_t^+}\frac{\tau_z-1}{\xi(\Za)-\xi(z)}.
\end{align*}

Summarising the above discussion, we have shown that
\begin{align}
 \label{e:u}
 U(t,y)  & = e^{t\xi(\Za)-2t}\prod_{i\in\mathfrak{I}^t_0}\frac{1}{\xi(\Za)-\xi(y_i)} \\
 & \nonumber \times \int \limits_{S_t^{\ell(y)-k(y)}}
\exp \Big\{\sum_{i\in\mathfrak{I}^t_1}x_i\big(\xi(y_i)-\xi(\Za)\big)\Big\}  \Phi_t\Big(\frac{1}{\Sigma_t^+}
\big[ t-\sum_{i\in\mathfrak{I}^t_1} x_i  -W_t(y)-M_t^+\big]\Big)   d({\bf x}\cap \mathfrak{I}^t_1) .
\end{align}
This representation is useful because we show, in Lemma \ref{l:normal} below, that $\Phi_t \to \Phi$ almost surely on~$\mathcal{E}_t$, where $\Phi$ is the distribution function of a standard normal random variable. \smallskip

It remains to define $\Xi_t(y)$ and show that \eqref{p:pq2} holds; to do this, we split the analysis into the non-critical regimes and the critical regime.
\smallskip

\emph{Non-critical regimes.}
In these regimes we define
\begin{align*}
\Xi_t(y)=e^{t\xi(\Za)-2t}\xi(\Za)^{-|\mathcal{K}_t^+|}
\int\limits_{S_t^{\ell(y)-k(y)}} & \exp\Big\{\sum_{i\in\mathfrak{I}^t_1} x_i\big(\xi(y_i)-\xi(\Za)\big)\Big\}
d({\bf x}\cap\mathfrak{I}^t_1),
\end{align*}
which is clearly $\mathcal{F}_t$-measurable as all values $\xi(z), z\in\mathcal{K}_t$, have been removed. Using the representation~\eqref{e:u} and observing that $W_t(y)=0$ since $w_t=0$, it suffices to define a (possibly random) scale $T_t \in [0, t]$ and show, on $\mathcal{E}_t$, that: (i)
\begin{align*}
\Phi_t \left( \frac{1}{\Sigma_t^+}
\big[ s  - M_t^+\big] \right) \to 1
\end{align*}
uniformly for all $y\in\mathcal{P}^t_+$ and $s\in [T_t,t]$; and (ii) the contribution to
\begin{align}
 \label{e:neg}
 \int\limits_{S_t^{\ell(y)-k(y)} } \exp\Big\{\sum_{i\in\mathfrak{I}^t_1} x_i\big(\xi(y_i)-\xi(\Za)\big)\Big\} d({\bf x}\cap\mathfrak{I}^t_1),
 \end{align}
from the domain $S^{\ell(y)-k(y)}_{T_t,t}=S^{\ell(y)-k(y)}_{t}\backslash S^{\ell(y)-k(y)}_{t-T_t}$ is negligible.
\smallskip
 
For the case in which $\eta(n)$ converges, we have that $\mathcal{K}_t^+=E \cap [0, \Za]$ is almost surely bounded by Lemma~\ref{l:kk}. Hence on $\mathcal{E}_t$, estimating $f_ta_t<\xi(\Za)-\xi(z)<g_ta_t$, we have 
\begin{align*}
\Gamma_t <  \frac{g_t}{f_t|\mathcal{K}_t^+|^{1/2}}\sum_{z\in \mathcal{K}_t^+}|\tau_z-1|
=\frac{g_t}{f_t}\cdot\frac{1}{|E \cap [0, \Za]|}\sum_{z\in E \cap [0, \Za]}|\tau_z-1|
\end{align*}
eventually, where the sum on the right-hand side is finite. Hence, setting $T_t=g_t^2/f_t$, and using the fact that, on $\mathcal{E}_t$ as $t \to \infty$,
\begin{align*}
M_t^+<\frac{|\mathcal{K}_t^+|}{a_tf_t} = o(T_t) 
\qquad\text{and}\qquad
(\Sigma_t^+)^2 < \frac{|\mathcal{K}_t^+|}{a_t^2f_t^2} \to 0
\end{align*} 
we have 
\begin{align*}
\Phi_t \left( \frac{1}{\Sigma_t^+}
\big[ s  - M_t^+\big] \right) \to 1
\end{align*}
uniformly for all $y\in\mathcal{P}^t_+$ and $s\in [T_t,t]$. It remains to show that the contribution to \eqref{e:neg} from $S^{\ell(y)-k(y)}_{T_t,t}$ is negligible; this will be done together with the next case.
\smallskip

For the case in which $\eta(n)\to\infty$, we instead set $T_t=|\mathcal{K}_t^+|/(a_tf_t^3)$ and use the fact that by Lemma~\ref{l:normal} below $\Phi_t\to \Phi$ almost surely on~$\mathcal{E}_t$. Since $\Phi$ is continuous the convergence is uniform, and as $\Phi(x)\to 1$ as $x\to\infty$, it is then sufficient to show that 
$\big[s-M_t^+\big]/\Sigma_t^+\to\infty$ uniformly for all $s\in [T_t, t]$, as well as to show that the contribution to \eqref{e:neg} from $S^{\ell(y)-k(y)}_{T_t,t}$ is negligible.
\smallskip

Observe that on $\mathcal{E}_t$ we have 
\begin{align*}
M_t^+<\frac{|\mathcal{K}_t^+|}{a_tf_t}=o(T_t)
\qquad\text{and}\qquad
 (\Sigma_t^+)^2<\frac{|\mathcal{K}_t^+|}{a_t^2f_t^2}.
\end{align*} 
Hence on event $\mathcal{E}_t$, for all $s\ge T_t$ we have
\begin{align*}
\frac{1}{\Sigma_t^+}\big[s-M_t^+\big]
&>\frac{|\mathcal{K}_t^+|^{1/2}}{2f_t^2}(1+o(1))\to\infty 
\end{align*}
almost surely, where we also make use of \eqref{fgk}.
\smallskip

Now let us show that the contribution to \eqref{e:neg} from $S^{\ell(y)-k(y)}_{T_t,t}$ is negligible in both cases. Integrating~\eqref{e:neg} with respect to the variables $x_i$, $i\in\mathfrak{I}^t_1\backslash
\mathfrak{I}^t_2$, corresponding to the visits to $\Za$ we obtain 
\begin{align}
\int\limits_{S_t^{\ell(y)-k(y)}} &\exp\Big\{\sum_{i\in\mathfrak{I}^t_1} x_i\big(\xi(y_i)-\xi(\Za)\big)\Big\}
d({\bf x}\cap\mathfrak{I}^t_1)\notag\\
&=\frac{1}{m(y)!}\int\limits_{S_t^{n(y)}} \Big(t-\sum_{i\in\mathfrak{I}^t_2} x_i\Big)^{m(y)}\exp\Big\{\sum_{i\in\mathfrak{I}^t_2} x_i\big(\xi(y_i)-\xi(\Za)\big)\Big\}
d({\bf x}\cap\mathfrak{I}^t_2).
\label{cc9}
\end{align}
On the other hand, integrating~\eqref{e:neg} just over $S^{\ell(y)-k(y)}_{T_t, t}$ we obtain 
\begin{align*}
&\int\limits_{S_{T_t,t}^{\ell(y)-k(y)}}
\exp\Big\{\sum_{i\in\mathfrak{I}^t_1}x_i\big(\xi(y_i)-\xi(\Za)\big)\Big\}
d({\bf x}\cap\mathfrak{I}^t_1)\\
& = \frac{1}{m(y)!}  \! \!
\int\limits_{S_t^{n(y)}}  \! \! \Big[\Big(t-\sum_{i\in\mathfrak{I}^t_2} x_i\Big)^{m(y)}-
\Big(\big(t-T_t-\sum_{i\in\mathfrak{I}^t_2} x_i\big)\vee 0\Big)^{m(y)}  \Big]\exp\Big\{\sum_{i\in\mathfrak{I}^t_2} x_i\big(\xi(y_i)-\xi(\Za)\big)\Big\}
d({\bf x}\cap\mathfrak{I}^t_2).
\end{align*} 
We will split this integral in the sum of two, corresponding to the domain of integration being $S^{n(y)}_{tf_t}$
and its complement. 
Observe that $T_t=o(t)$ in all cases, which is obvious if $\eta$ converges and follows from 
\begin{align*}
T_t<\frac{g_t\eta(r_t)}{\theta_t^{\alpha}a_tf_t^3}
<\frac{g_t r_t}{a_tf_t^3}=o(t)
\end{align*}
otherwise. Recalling the definition $\Lambda_t= 9et /\xi(\Za)$ and that $y$ makes at most $\Lambda_t$ returns to $\Za$,  on the domain $S^{n(y)}_{tf_t}$ we can therefore use Bernoulli's inequality to obtain eventually
\begin{align*}
\Big(t-\sum_{i\in\mathfrak{I}^t_2} x_i\Big)^{m(y)}-
\Big(\big(t-T_t-\sum_{i\in\mathfrak{I}^t_2} x_i\big)\vee0\Big)^{m(y)} &= \Big(t-\sum_{i\in\mathfrak{I}^t_2} x_i\Big)^{m(y)}-
\Big(t-T_t-\sum_{i\in\mathfrak{I}^t_2} x_i\Big)^{m(y)} \\
& \le \Big(t-\sum_{i\in\mathfrak{I}^t_2} x_i\Big)^{m(y)}
\frac{T_t\Lambda_t}{t-tf_t}.
\end{align*}
Observe that by~\eqref{fg2c} and since $\eta(r_t) \ll r_t$ and $r_t=a_t^{\alpha}$ we have, on $\mathcal{E}_t$,
\begin{align*}
\frac{T_t\Lambda_t}{t-tf_t}<\frac{g_t^2 \eta(r_t)}{\theta_t^{\alpha}a_t^2f_t^4}
=\left\{\begin{array}{lll}
\displaystyle
\frac{g_t^2}{f_t^4}\cdot\frac{\eta(r_t)}{r_t^{2/\alpha}} & \to 0 & \text{ if }1\ll  \eta(n)\ll \kappa(n),\\
\displaystyle
\frac{g_t^2}{f_t^{4+\alpha}} \cdot\Big[\frac{r_t^{2/\alpha}}{\eta(r_t)}\Big]^{\frac{2}{\alpha-2}}&\to 0 & \text{ if }\eta(n)\gg \kappa(n) \ \text{and} \ \alpha > 2 , \\
\displaystyle
\frac{g_t^2}{f_t^{6}} \cdot  \frac{\eta(r_t)}{r_t} \exp \left( - \log r_t + \frac{2 r_t}{\eta(r_t) f_t}   \right) &\to 0 & \text{ if }\eta(n)\gg \kappa(n) \ \text{and} \ \alpha = 2 , \\
\end{array}
\right.
\end{align*}
and that this ratio also obviously converges to zero if $\eta(n)$ converges. This implies that in all non-critical regimes 
\begin{align*}
\int\limits_{S_{tf_t}^{n(y)}}& \Big[\Big(t-\sum_{i\in\mathfrak{I}^t_2} x_i\Big)^{m(y)}-
\Big(t-T_t-\sum_{i\in\mathfrak{I}^t_2} x_i\Big)^{m(y)}\Big]\exp\Big\{\sum_{i\in\mathfrak{I}^t_2} x_i\big(\xi(y_i)-\xi(\Za)\big)\Big\}
d({\bf x}\cap \mathfrak{I}^t_2)\\
&=o(1)\int\limits_{S_t^{n(y)}} \Big(t-\sum_{i\in\mathfrak{I}^t_2} x_i\Big)^{m(y)}\exp\Big\{\sum_{i\in\mathfrak{I}^t_2} x_i\big(\xi(y_i)-\xi(\Za)\big)\Big\}
d({\bf x}\cap \mathfrak{I}^t_2).
\end{align*}
Combining this with~\eqref{cc9}, it remains to show the the integral on the left-hand side of the above formula, taken over $S_{t}^{n(y)}\backslash S_{tf_t}^{n(y)}$, is also negligible with respect to the integral on the right-hand side. Estimating the expression in the square brackets by the  first term, we can easily see that this follows from Lemma~\ref{l:negl} below. 
\smallskip

\emph{Critical regime.} In the critical regime on $\mathcal{E}_t$ we have 
\begin{align}
\label{www}
0<W_t(y)\le \frac{1}{a_tf_t}\sum_{i=1}^{w(y)}\hat \tau_i
\le \frac{w_t}{a_tf_t}\hat\tau_{\lceil w_t \rceil}^{\ssup 1}
\le \frac{w_t^2}{a_tf_t}\to 0
\end{align}
uniformly in $y$ almost surely, where 
$\hat \tau_{\lceil w_t \rceil}^{\ssup 1}$ denotes the maximum of $\hat \tau_1, \dots, \hat\tau_{\lceil w_t \rceil}$, which is bounded eventually by $w_t$ almost surely. We henceforth assume that the event $\mathcal{E}_t^{cr}$ holds; this is valid by Proposition \ref{p:e2}. The laws of large numbers for $M_t^+$ and $\Sigma_t^+$ specified in the event $\mathcal{E}_t^{cr}$ together with Lemma~\ref{l:normal} below suggest that we should define 
\begin{align*}
\Xi_t(y)=e^{t\xi(\Za)-2t}\xi(\Za)^{-|\mathcal{K}_t^+|-w(y)}
\int\limits_{S_t^{\ell(y)-k(y)}} &
\exp\Big\{\sum_{i\in\mathfrak{I}^t_1} x_i\big(\xi(y_i)-\xi(\Za)\big)\Big\}\\
&\times \Phi\Big(\frac{1}{\bar{S}_t}
\Big[t-\sum_{i\in\mathfrak{I}^t_1} x_i- \bar{M}_t\Big]\Big)
d({\bf x}\cap \mathfrak{I}^t_1). 
\end{align*}
Again, it is easy to see that $\Xi_t(y)$ is $\mathcal{F}_t$-measurable. Using the representation in \eqref{e:u}, we therefore have that 
\begin{align*}
\Upsilon_t(y)=\frac{\Theta_{t,1}(y)}{\Theta_{t,2}(y)},
\end{align*}
where
\begin{align*}
\Theta_{t,1}(y)=\!\!\!\!\!\int\limits_{S_t^{\ell(y)-k(y)}} \!\!\! & 
\!\!\exp\Big\{\sum_{i\in\mathfrak{I}^t_1} x_i\big(\xi(y_i)-\xi(\Za)\big)\Big\}  \Phi_t\Big(\frac{1}{\Sigma_t^+}
\big[t-\sum_{i\in\mathfrak{I}^t_1} x_i-W_t(y)-M_t^+\big]\Big)
d({\bf x}\cap \mathfrak{I}^t_1) ,\\
\Theta_{t,2}(y)=\!\!\!\!\!\int\limits_{S_t^{\ell(y)-k(y)}}\!\!\!&  
\!\!\exp\Big\{\sum_{i\in\mathfrak{I}^t_1}  x_i\big(\xi(y_i)-\xi(\Za)\big)\Big\} \, \Phi\Big(\frac{1}{\bar{S}_t}\,
\Big[t-\sum_{i\in\mathfrak{I}^t_1} x_i- \bar{M}_t\Big]\Big)
d({\bf x}\cap \mathfrak{I}^t_1).
\end{align*}
Hence it suffices to show that, uniformly in $y\in \mathcal{P}^t_+$, 
\begin{align}
\label{Theta1}
|\Theta_{t,1}(y)-\Theta_{t,2}(y)|
&=\Delta_{t,1}\int\limits_{S_t^{\ell(y)-k(y)}}  
\exp\Big\{\sum_{i\in\mathfrak{I}^t_1} x_i\big(\xi(y_i)-\xi(\Za)\big)\Big\}
d({\bf x}\cap \mathfrak{I}^t_1),\\
\Theta_{t,2}(y)&\ge \Delta_{t,2} \int\limits_{S_t^{\ell(y)-k(y)}} 
\exp\Big\{\sum_{i\in\mathfrak{I}^t_1} x_i\big(\xi(y_i)-\xi(\Za)\big)\Big\}
d({\bf x}\cap \mathfrak{I}^t_1),
\label{Theta2}
\end{align}
where $\Delta_{t,1}$ converges to zero almost surely on $\mathcal{E}_t\cap\mathcal{E}_t^{cr}$,
and $\Delta_{t,2}$ is bounded away from zero in probability. These properties guarantee that \begin{align*}
\frac{\Delta_{t,1}}{\Delta_{t,2}}\to 0
\end{align*} 
in probability.
\smallskip 

For each $t>0$, $s\in [0,t]$, and $y\in \mathcal{P}^t_+$, denote  
\begin{align*}
X_t(s,y)=\frac{1}{\Sigma_t^+}\big[s-W_t(y)-M_t^+\big]
\qquad\text{and}\qquad
\hat X_t(s)=\frac{1}{\bar{S}_t}\Big[s- \bar{M}_t\Big].
\end{align*}
Observe that the properties specified in 
$\mathcal{E}_t^{cr}$ imply the additional properties 
\begin{align*}
M_t^+>g_ta_t
\qquad\text{and}\qquad f_t<\frac{1}{\Sigma_t^+}<g_t
\end{align*}
eventually. 
Using~\eqref{www} we have on $\mathcal{E}\cap \mathcal{E}_t^{cr}$ 
\begin{align}
|X_t(s,y)-\hat X_t(s)|
&\le s\Big|\frac{1}{\Sigma_t^+}-\frac{1}{\bar{S}_t}\Big|
+M_t \Big|\frac{1}{\Sigma_t^+}
-\frac{1}{\bar{S}_t}\big|
+\frac{1}{\Sigma_t^+}|M_t^+-\bar{M}_t\big|
+\frac{W_t(y)}{\Sigma_t^+}\notag\\
&\le \frac{s g_t}{a_t^2}+\frac{g_t^2}{a_t}+\frac{g_t\lambda(r_t)^{1/2}}{a_t}+\frac{g_tw_t^2}{f_ta_t}\to 0
\label{cc1}
\end{align}
uniformly for all $s\in [0,a_t^{3/2}]$. Further, using~\eqref{www} we have for all $s\in \big[a_t^{3/2},t\big]$ on
$\mathcal{E}\cap \mathcal{E}_t^{cr}$
\begin{align}
\label{cc2}
X_t(s,y)\ge f_t\Big[a_t^{3/2}-\frac{w_t^2}{a_tf_t}-g_ta_t\Big]\to\infty
\qquad\text{and}\qquad
\hat X_t(s)\ge f_t\Big[a_t^{3/2}-g_ta_t\Big]\to\infty,
\end{align}
with the  convergences being uniform in $y\in\mathcal{P}^t_+$.
By Lemma~\ref{l:normal}, $\Phi_t\to \Phi$, and the convergence is uniform since $\Phi$ is continuous. Using also 
that $\Phi$ is uniformly continuous we obtain from~\eqref{cc1} and \eqref{cc2} that 
\begin{align*}
\Delta_{t,1}=\max_{s\in [0,t]}\max_{y\in \mathcal{P}^t_+}\big|\Phi_t(X_t(s,y))-\Phi(\hat X_t(s))\big|\to 0
\end{align*}
almost surely on $\mathcal{E}_t\cap \mathcal{E}_t^{cr}$,
which implies~\eqref{Theta1}.
\smallskip

To prove~\eqref{Theta2}, we  first 
observe that by integrating with respect to $x_i$, $i\in \mathfrak{I}^t_1\backslash \mathfrak{I}^t_2$
corresponding to the extra visits to $\Za$, we have
\begin{align}
&\int\limits_{S_t^{\ell(y)-k(y)}} 
\exp\Big\{\sum_{i\in\mathfrak{I}^t_1} x_i\big(\xi(y_i)-\xi(\Za)\big)\Big\}
d({\bf x}\cap \mathfrak{I}^t_1)\notag\\
&=\frac{1}{m(y)!}\int\limits_{S_t^{n(y)}} 
\Big(t-\sum_{i\in\mathfrak{I}^t_2} x_i\Big)^{m(y)}
\exp\Big\{\sum_{i\in\mathfrak{I}^t_2} x_i\big(\xi(y_i)-\xi(\Za)\big)\Big\}
d({\bf x}\cap \mathfrak{I}^t_2)).
\label{cc7}
\end{align}
Restricting the integral in $\Theta_{t,2}(y)$ to the domain where the argument of $\Phi$ is
positive and, similarly, integrating with respect to $x_i$, $i\in \mathfrak{I}^t_1\backslash \mathfrak{I}^t_2$, we obtain   
\begin{align}
\Theta_{t,2}(y)
&\ge \frac 1 2 
 \int\limits_{S_{t-\bar{M}_t}^{\ell(y)-k(y)}} \exp\Big\{\sum_{i\in\mathfrak{I}^t_1} x_i\big(\xi(y_i)-\xi(\Za)\big)\Big\}
d({\bf x}\cap \mathfrak{I}^t_1)\notag\\
& =\frac{1}{2m(y)!} 
 \int\limits_{S_{t-\bar{M}_t}^{n(y)}} \Big(t-\bar{M}_t-\sum_{i\in\mathfrak{I}^t_2}x_i\Big)^{m(y)} \exp\Big\{\sum_{i\in\mathfrak{I}^t_2} x_i\big(\xi(y_i)-\xi(\Za)\big)\Big\}
d({\bf x}\cap \mathfrak{I}^t_2)).
\label{cc8}
\end{align}
Hence in order to prove~\eqref{Theta2} we need to show that the integral in~\eqref{cc7} is lower-bounded by the integral in~\eqref{cc8} multiplied by some $\Delta_{t,2}$ with the required properties. 
\smallskip

Observe that on $\mathcal{E}_t^{cr}$ we have $\bar{M}_t<g_ta_t=o(t)$
as $\alpha\ge2$,  which implies $t-\bar{M}_t>tf_t$. We will restrict the integral in~\eqref{cc8} to the even smaller domain
$S_{tf_t}^{n(y)}$, where we can estimate  
\begin{align*}
\Big(t-\bar{M}_t-\sum_{i\in\mathfrak{I}^t_2}x_i\Big)^{m(y)}
&= \Big(t-\sum_{i\in\mathfrak{I}^t_2}x_i\Big)^{m(y)}
\Big(1-\bar{M}_t\big/\big[t-
\sum_{i\in\mathfrak{I}^t_2} x_i\big]\Big)^{m(y)} \\
&\ge  \Big(t-\sum_{i\in\mathfrak{I}^t_2}x_i\Big)^{m(y)}\Big( 1-\frac{2\bar{M}_t}{t}\Big)^{\Lambda_t}.
\end{align*}
Since 
\begin{align*}
\Big( 1-\frac{2 \bar{M}_t}{t}\Big)^{\Lambda_t}=\exp\Big\{-\frac{18eN(\Za)}{\xi(\Za)^2}(1+o(1))\Big\}
\end{align*}
we obtain 
\begin{align*}
\Big(t- \bar{M}_t-\sum_{i\in\mathfrak{I}^t_2}x_i\Big)^{m(y)}
\ge 4\Delta_{t,2}\Big(t-\sum_{i\in\mathfrak{I}^t_2}x_i\Big)^{m(y)}
\end{align*}
where 
\begin{align*}
\Delta_{t,2}=\frac 1 4 \exp\Big\{-\frac{19eN(\Za)}{\xi(\Za)^2}\Big\}
=\frac 1 4 \exp\Big\{-19e\cdot \frac{N(\Za)}{a_t^2}\cdot\frac{a_t^2}{\xi(\Za)^2}\Big\}
\end{align*}
is bounded away from zero in probability by Proposition~\ref{LB1} and 
Lemma~\ref{l:etacr}.
This implies  
\begin{align}
\Theta_{t,2}(y)\ge \frac{2\Delta_{t,2}}{m(y)!} 
\int\limits_{S_{tf_t}^{n(y)}} \Big(t-
\sum_{i\in\mathfrak{I}^t_2} x_i\Big)^{m(y)} \exp\Big\{ \sum_{i\in\mathfrak{I}^t_2} x_i\big(\xi(y_i)-\xi(\Za)\big)\Big\}
d({\bf x}\cap \mathfrak{I}^t_2)).
\label{qqq}
\end{align}
It remains now to prove that this expression is of the same order as the expression on the left-hand side of~\eqref{cc7}. Comparing it with~\eqref{qqq}, it suffices to show that the integral over $S_{t}^{n(y)}\backslash S_{tf_t}^{n(y)}$ is negligible with respect to the integral over $S_{tf_t}^{n(y)}$ uniformly for all $y\in \mathcal{P}^t_+$, which follows from Lemma~\ref{l:negl} below.
\end{proof}

Before completing the proof of Proposition \ref{p:pq}, we establish the three lemmas that were used in the previous proof. Note that throughout we recall all notation used in that proof.

\begin{lemma} 
\label{l:normal}
Assume $\eta(n)\to\infty$. As $t \to \infty$, $\Phi_t\to \Phi$ almost surely on $\mathcal{E}_t$, where $\Phi$ is the distribution function of the standard normal distribution. 
\end{lemma}

\begin{proof} 
Denote
\begin{align*}
V_t(z)=\frac{1}{\Sigma_t^+}\cdot\frac{\tau_z-1}{\xi(\Za)-\xi(z)}.
\end{align*}
Since $\mathrm{E}_{\mathcal{F}_t}V_t(z)=0$ and
\begin{align*}
\sum_{z\in \mathcal{K}_t^+} \mathrm{E}_{\mathcal{F}_t} V_t^2(z)=1,
\end{align*}
the statement will follow from~\cite[Thm.\ B1]{MPS} provided the Lindenberg condition is satisfied.  
Given $\e>0$, we have on $\mathcal{E}_t$, using $\Sigma_t^+\ge |\mathcal{K}_t^+|^{1/2}(a_tg_t)^{-1} $ 
\begin{align*}
\sum_{z\in \mathcal{K}_t^+}\mathrm{E}_{\mathcal{F}_t}
\big[V_t^2(z)\one\{|V_t(z)|\ge \e\}\big]
& \le \mathrm{E}_{\mathcal{F}_t}
\Big[(\tau_z-1)^2\one\Big\{|\tau_z-1|\ge \frac{\e  f_t |\mathcal{K}_t^+|^{1/2}}{g_t}\Big\}\Big]\to 0
\end{align*}
almost surely, where we also make use of \eqref{fgk}.
\end{proof}

\begin{lemma} As $t\to\infty$,
\label{l:negl}
\begin{align*}
&\int\limits_{S_{t f_t}^{n(y)}}
\Big(t-\sum_{i\in\mathfrak{I}^t_2}x_i\Big)^{m(y)}
\exp\Big\{\sum_{i\in\mathfrak{I}^t_2}x_i\big(\xi(y_i)-\xi(\Za)\big)\Big\}
d({\bf x}\cap \mathfrak{I}^t_2)\\
&\quad  \sim  \int\limits_{S_t^{n(y)}}
\Big(t-\sum_{i\in\mathfrak{I}^t_2}x_i\Big)^{m(y)}
\exp\Big\{\sum_{i\in\mathfrak{I}^t_2}x_i\big(\xi(y_i)-\xi(\Za)\big)\Big\}
d({\bf x}\cap \mathfrak{I}^t_2)
\end{align*}
uniformly for all $y\in\mathcal{P}^t_{0}$ almost surely on $\mathcal{E}_t$.
\end{lemma}

\begin{proof}
First, estimating on $S_{t}^{n(y)}\backslash S_{tf_t}^{n(y)}$ we have 
\begin{align*}
\sum_{i\in\mathfrak{I}^t_2} x_i\big(\xi(y_i)-\xi(\Za)\big)
\le -a_tf_t \sum_{i\in\mathfrak{I}^t_2} x_i \le -t a_tf^2_t
\end{align*}
and so 
\begin{align*}
\int\limits_{S_{t}^{n(y)}\backslash S_{tf_t}^{n(y)}}\Big(t-
\sum_{i\in\mathfrak{I}^t_2}x_i\Big)^{m(y)} &\exp\Big\{ \sum_{i\in\mathfrak{I}^t_2} x_i\big(\xi(y_i)-\xi(\Za)\big)\Big\}
d({\bf x}\cap \mathfrak{I}^t_2)\\
&\le \exp\big\{-ta_tf_t^2\big\}t^{m(y)}\frac{t^{n(y)}}{n(y)!}=:\mathfrak{L}_t(y).
\end{align*}
Second, estimating on $S_{t/\log t}^{n(y)}$ we have 
\begin{align*}
\sum_{i\in\mathfrak{I}^t_2} x_i\big(\xi(y_i)-\xi(\Za)\big)
\ge -\xi(\Za) \sum_{i\in\mathfrak{I}^t_2} x_i \ge -\frac{t a_tg_t}{\log t}
\end{align*}
and so 
\begin{align*}
\int\limits_{S_{t/\log t}^{n(y)}}\!\!\! \!\Big(t-
\sum_{i\in\mathfrak{I}^t_2} x_i\Big)^{m(y)}& \exp\Big\{ \sum_{i\in\mathfrak{I}^t_2} x_i\big(\xi(y_i)-\xi(\Za)\big)\Big\}
d({\bf x}\cap \mathfrak{I}^t_1)\\
&\ge \exp\Big\{-\frac{ta_tg_t}{\log t}\Big\}\Big(t-\frac{t}{\log t}\Big)^{m(y)}\frac{1}{n(y)!}\Big(\frac{t}{\log t}\Big)^{n(y)}=:\mathfrak{U}_t(y).
\end{align*}
Finally, using $n(y)\le \ell(y)\le R_t\le 2r_tg_t$ and $m(y)\le \Lambda_t$ we get 
\begin{align*}
\frac{\mathfrak{L}_t(y)}{\mathfrak{U}_t(y)}
&=\exp\Big\{-ta_tf_t^2+\frac{ta_tg_t}{\log t}
-m(y)\log\Big(1-\frac{1}{\log t}\Big)+n(y)\log\log t\Big\}\\
&\le \exp\Big\{-ta_tf_t^2+\frac{ta_tg_t}{\log t}
-\Lambda_t\log\Big(1-\frac{1}{\log t}\Big)+2r_tg_t\log\log t\Big\}
\to 0
\end{align*}
uniformly in $y$ since the first term dominates the rest. 
\end{proof}

\begin{lemma} 
\label{l:etacr}
In the critical regime, as $t\to\infty$,
\begin{align*}
\frac{\lambda(r_t)N(\Za)}{a_t^2}\Rightarrow 
\beta (X^{\ssup 1})^{\frac{2}{\alpha}}.
\end{align*}
\end{lemma}

\begin{proof} Since $r_t=a_t^{\alpha}$ we have for the case $\alpha>2$,
\begin{align*}
\frac{N(\Za)}{a_t^2}=\frac{N(\Za)}{\eta(\Za)}\cdot \frac{\eta(\Za)}{(\Za)^{2/\alpha}}\cdot \Big[\frac{\Za}{r_t}\Big]^{\frac{2}{\alpha}}
\Rightarrow \beta (X^{\ssup 1})^{\frac{2}{\alpha}}
\end{align*}
by Proposition~\ref{LB1} and Lemma~\ref{l:nn}. For the case $\alpha=2$, we instead have
\[ \frac{\lambda(r_t) N(\Za)}{a_t^2}=\frac{N(\Za)}{\eta(\Za)}\cdot \frac{\eta(\Za)}{\Za/\log \Za}\cdot \frac{\Za\log r_t}{r_t\log\Za}\Rightarrow \beta X^{\ssup 1}, \]
again by Proposition~\ref{LB1} and Lemma~\ref{l:nn}. 
\end{proof}

We conclude the section by completing the proof of Proposition \ref{p:pq}, which follows easily from Proposition~\ref{p:ext} above.

\subsection{Proof of Proposition \ref{p:pq}}
Let $\Xi_t(y)$ be defined as in Proposition \ref{p:ext}, and denote
\begin{align}
\label{pp}
P_t= - \log \sum_{y\in \mathcal{P}^t_+}\Xi_t(y) + \log \sum_{y\in \mathcal{P}^t_-}\Xi_t(y) .
\end{align}
Note that this immediately gives that $P_t=0$ in the subcritical and critical regimes since $\Xi_t$ is then symmetric by Proposition~\ref{p:ext}.
\smallskip

For each path $y\in\mathcal{P}^t_{0}$, denote by $\mathfrak{I}^t(y)$ the set of all indices $i$ such that 
$y_i\in \mathcal{K}_t$ and $y_i$ is not the first visit to this particular point of $\mathcal{K}_t$. 
Combining Lemmas \ref{L:nullpaths} and \ref{l:null0} and Proposition~\ref{p:ext} we obtain
\begin{align*}
u(t,\Za)=(1+o(1))
\prod_{z\in\mathcal{K}_t^+}\Big(1-\frac{\xi(z)}{\xi(\Za)}\Big)^{-1}
\sum_{y\in \mathcal{P}^t_+}\Xi_t(y)\prod_{i\in \mathfrak{I}^t(y)}\Big(1-\frac{\xi(y_i)}{\xi(\Za)}\Big)^{-1}.
\end{align*}
where the $o(1)$ term tends to zero almost surely on the event $\mathcal{E}_t$ in the non-critical regimes, and in probability in the critical regime.
\smallskip

In the non-critical regimes $w_t=0$ and $\mathfrak{I}^t(y)=\varnothing$. Hence the second product equals one, and the sum is $\mathcal{F}_t$-measurable. Since the logarithm of the first product is by definition $Q_t^+$, and using the analogous argument for $u(t,-\Za)$, we obtain the result in the non-critical regimes.
\smallskip

In the critical regime we have, using Lemma~\ref{l:zeta} (and recalling the definition of $\zeta_t$ from that lemma), 
\begin{align*}
0<\log \prod_{i\in \mathfrak{I}^t(y)}\Big(1-\frac{\xi(y_i)}{\xi(\Za)}\Big)^{-1}
&\le-w_t\log\Big(1-\frac{\zeta_t}{\xi(\Za)}\Big)=(1+o(1))\frac{w_t}{f_t}a_t^{2/\alpha-1}\lambda(t)^{-1/2}\to 0
\end{align*}
in probability, which implies the result. 
\bigskip


\section{Typical properties}
\label{sec:typ}

We next establish that the events $\mathcal{E}_t$ and $\mathcal{E}_t^\text{cr}$ hold eventually with overwhelming probability, and in particular to complete the proof of Propositions~\ref{p:e} and \ref{p:e2}. Most of the properties in $\mathcal{E}_t$ and $\mathcal{E}_t^\text{cr}$ follow automatically from the point process machinery we developed in Section \ref{sec:pp}, although some need to established by more direct methods; we do this in the following three lemmas.

\begin{lemma}
\label{l:gap}
As $t\to\infty$, 
\begin{align*}
\Prob\big(\big[\Za-\alpha,\Za+\alpha\big]\cap \N\subset D\big)\to 1.
\end{align*}
\end{lemma}

\begin{proof} 
We have 
\begin{align*}
\Prob\big(\big[\Za-\alpha,\Za+\alpha\big]\cap \N\not\subset D\big)\le \Prob\big(\big[Z_t^{\ssup{1*}}-\alpha,Z_t^{\ssup{1*}}+\alpha\big]\cap \N\not\subset D\big)+\Prob\big(\Za\neq Z_t^{\ssup{1*}}\big).
\end{align*}
The second probability tends to zero by Proposition~\ref{LB1}. To show that the first probability also tends to zero, denote by 
$\mathcal{G}$ the $\sigma$-algebra generated by $\{\xi(z):z\in \N_0\}$, and denote the conditional probability with respect to $\mathcal{G}$ by $\Prob_{\mathcal{G}}$. Since $Z_t^{\ssup{1*}}$ is $\mathcal{G}$-measurable and independent of $D$, we have 
\begin{align*}
\Prob_{\mathcal{G}}\big(\big[Z_t^{\ssup{1*}}-\alpha,Z_t^{\ssup{1*}}+\alpha\big]\cap \N\not\subset D\big)=1-\!\!\!\!\!\!\!\!\prod_{|z-Z_t^{\ssup{1*}}|<\alpha}\!\!\!\!\!\!\!\!p(z)  \le 1- \max_{ |z-Z_t^{\ssup{1*}}|<\alpha } p(z)^{2\lfloor\alpha\rfloor+1}\to 0
\end{align*}
almost surely since $Z_t^{\ssup{1*}}\to \infty$ almost surely and $p \to 1$. Hence 
\begin{align*}
 \Prob\big(\big[Z_t^{\ssup{1*}}-\alpha,Z_t^{\ssup{1*}}+\alpha\big]\cap \N\not\subset D\big)\to 0
\end{align*}
by the dominated convergence theorem.
\end{proof}

\begin{lemma}
\label{l:nohigh}
As $t\to\infty$, 
\begin{align*}
\Prob\big(2\xi(z)<\xi(\Za) \text{ \rm for all }\, 0<|z-\Za|\le \alpha\big)\to 1. 
\end{align*}
\end{lemma}

\begin{proof}
We have 
\begin{align*}
\Prob&\big(2\xi(z)<\xi(\Za) \text{ for all }\, 0<|z-\Za|\le \alpha\big)\\
&\le \Prob\big(2\xi(z)<\xi(Z_t^{\ssup{1*}}) \text{ for all }\, 0<|z-Z_t^{\ssup{1*}}|\le \alpha\big)+\Prob\big(\Za\neq Z_t^{\ssup{1*}}\big).
\end{align*}
The second probability tends to zero by Proposition~\ref{LB1}. To show that the first probability also tends to zero, observe that, conditionally on $Z_t^{\ssup{1*}}$, $\xi(z)$, $z\neq Z_t^{\ssup{1*}}$, are independent and have Pareto distribution with parameter $\alpha$ conditioned on $\Psi_t(z)<\Psi_t(Z_t^{\ssup{1*}})$. Since $\Psi_t(z)<\xi(z)$ we obtain  
\begin{align*}
\Prob\big(2\xi(z)\ge \xi(Z_t^{\ssup{1*}})\big|Z_t^{\ssup{1*}},\xi(Z_t^{\ssup{1*}})\big)
\le \frac{\xi(Z_t^{\ssup{1*}})^{\alpha}}{2^{\alpha}(1-\xi(Z_t^{\ssup{1*}})^{\alpha})}\to 0
\end{align*}
almost surely by Proposition~\ref{LB1}. This implies that 
\begin{align*}
\Prob\big(2\xi(z)<\xi(Z_t^{\ssup{1*}}) \text{ for all }\, 0<|z-Z_t^{\ssup{1*}}|\le \alpha\big|Z_t^{\ssup{1*}},\xi(Z_t^{\ssup{1*}})\big)\to 0
\end{align*}
almost surely, which yields the result by the dominated convergence theorem.
\end{proof}

\begin{lemma} 
\label{l:cr}
In the critical regime, as $t\to\infty$,
\begin{align*}
a_t|M_t^+ - \bar{M}_t\big|
\qquad\text{and}\qquad
a_t^2\Big|\frac{1}{\Sigma_t^+}-\frac{1}{\bar{S}_t}\Big|
\end{align*}
are bounded in probability.
\end{lemma}

\begin{proof} 
We first establish the result conditionally on the $\sigma$-algebra $\mathcal{F}_t$; the full result then holds unconditionally by the dominated convergence theorem. \smallskip

Observe that, conditionally on $\mathcal{F}_t$, $\xi(z)$, $z\in\mathcal{K}_t^+$, are independent and distributed as Pareto random variables with parameter~$\alpha$ (for this recall that $\Za$ is defined as the maximiser of $\Psi_t$ on $D$). Recall also the fact that $|\mathcal{K}_t^+|=N(\Za)$ in the critical regime. The central limit theorem then implies that 
\begin{align}
\label{cltx}
\Delta_t = \frac{1}{\sigma \left( N(\Za) \, \lambda( N(\Za) ) \right) ^{1/2}}\Big[\sum_{z\in\mathcal{K}_t^+}\xi(z)- \gamma N(\Za) \Big]  
\end{align} 
converges in law to a standard normally distributed random variable, where $\sigma$ is defined as in Theorem~\ref{t:main1} (for the case $\alpha = 2$ see, e.g., \cite[Ex.\ 3.4.8]{dur}). Further, in the case $\alpha > 2$, by the strong law of large numbers, 
\begin{align}
\label{lln}
\sum_{z\in\mathcal{K}_t^+}\xi(z)^2<2 \hat\gamma\, N(\Za),
\end{align}
where $\hat \gamma = \rm{E} \xi(0)^2 = \alpha/(\alpha-2)$, whereas in the case $\alpha=2$ (see, e.g., \cite[Thm.\ 2.2.6]{dur})
\begin{align}
\label{lln2}
\frac1{N(\Za) \log N(\Za) } \sum_{z\in\mathcal{K}_t^+}\xi(z)^2  
\end{align}
is bounded above in probability.
\smallskip

Recall the quantity $\zeta_t$ from Lemma \ref{l:zeta}, and observe that for each $z\in\mathcal{K}_t^+$ we have   
\begin{align*}
\frac{\xi(z)}{\xi(\Za)}\le \frac{\zeta_t}{\xi(\Za)}
=\begin{cases}\frac{\zeta_t}{a_t^{2/\alpha}}\cdot\frac{a_t}{\xi(\Za)}\cdot a_t^{2/\alpha-1}&\to 0\quad\mbox{ if }\alpha>2,\\
\frac{\zeta_t(\log t)^{1/2}}{a_t}\cdot\frac{a_t}{\xi(\Za)}\cdot(\log t)^{-1/2}&\to0\quad\mbox{ if }\alpha=2,
\end{cases}
\end{align*}
almost surely uniformly for all $z\in\mathcal{K}_t^+$ by Proposition~\ref{LB1} and Lemma~\ref{l:zeta}. Hence we can use
\begin{align*}
1+x<\frac{1}{1-x}<1+x+2x^2,
\end{align*}
which holds for all sufficiently small $x$,
to obtain 
\begin{align*}
\frac{N(\Za)}{\xi(\Za)}+\frac{1}{\xi(\Za)^2}
\sum_{z\in\mathcal{K}_t^+}\xi(z)<
M_t^+
&<\frac{N(\Za)}{\xi(\Za)}+\frac{1}{\xi(\Za)^2}
\sum_{z\in\mathcal{K}_t^+}\xi(z)+\frac{2}{\xi(\Za)^3}
\sum_{z\in\mathcal{K}_t^+}\xi(z)^2
\end{align*}
eventually almost surely. We therefore have
\[ a_t|M_t^+-\bar{M}_t|
<\frac{\sigma a_t \left( N(\Za) \lambda(N(\Za) ) \right)^{1/2} |\Delta_t| }{\xi(\Za)^2}
+\frac{2 a_t N(\Za)  \lambda(N(\Za) )\sum_{z\in\mathcal{K}_t^+}\xi(z)^2}{\xi(\Za)^3}. \]

For the case $\alpha>2$, we combine this with~\eqref{lln} to obtain  
\begin{align*}
a_t|M_t^+ - \bar{M}_t|
&<\frac{\sigma a_t N(\Za)^{1/2} |\Delta_t|}{\xi(\Za)^2}
+\frac{4\hat\gamma a_t N(\Za)}{\xi(\Za)^3}\\
&=\sigma |\Delta_t| \cdot\frac{a_t^2 }{\xi(\Za)^2}
\cdot\frac{N(\Za)^{1/2}}{a_t}
+4\hat\gamma\cdot \frac{ a_t^3}{\xi(\Za)^3}\cdot \frac{N(\Za)}{a_t^2},
\end{align*}
which is bounded in probability by~\eqref{cltx}, Proposition~\ref{LB1}, and Lemma~\ref{l:etacr}. For the case $\alpha=2$, we instead have
\begin{align*}
a_t|M_t^+-\bar{S}_t|
& < |\Delta_t| \cdot\frac{a_t^2 }{\xi(\Za)^2}
\cdot\frac{(N(\Za)\log N(\Za))^{1/2}}{a_t}
\\&\phantom{=}+2\frac{ a_t^3}{\xi(\Za)^3}\cdot \frac{N(\Za)\log N(\Za)}{a_t^2}\cdot\frac{\sum_{z\in\mathcal{K}_t^+}\xi(z)^2}{N(\Za)\log N(\Za)},
\end{align*}
which is bounded in probability by~\eqref{cltx}, \eqref{lln2}, Proposition~\ref{LB1}, and Lemma~\ref{l:etacr}.
\smallskip

Similarly, using 
\begin{align*}
1+2x<\frac{1}{(1-x)^2}<1+2x+4x^2,
\end{align*}
which holds for all sufficiently small $x$, we obtain  
\begin{align}
\label{iii3}
\frac{N(\Za)}{\xi(\Za)^2}+\frac{2}{\xi(\Za)^3}\sum_{z\in \mathcal{K}_t^+}\xi(z)<(\Sigma_t^+)^2<\frac{N(\Za)}{\xi(\Za)^2}+\frac{2}{\xi(\Za)^3}\sum_{z\in \mathcal{K}_t^+}\xi(z)+\frac{4}{\xi(\Za)^4}\sum_{z\in \mathcal{K}_t^+}\xi(z)^2.
\end{align}
Denote 
\begin{align*}
A_t^2=\left(\xi(\Za)^{-2}+2\gamma\xi(\Za)^{-3}\right)N(\Za)\lambda(N(\Za))
\qquad\text{and}\qquad
B_t=a_t^2\big((\Sigma_t^+)^2-A_t^2\big).
\end{align*}
Observe that $A_t$ is bounded in probability by Proposition~\ref{LB1}, Lemma~\ref{l:nn} and Lemma~\ref{l:etacr} since 
\begin{align*}
A_t^2<2\cdot\frac{N(\Za)\lambda(N(\Za))}{a_t^2}\cdot\frac{a_t^2}{\xi(\Za)^2}
\end{align*}
eventually almost surely. It follows from~\eqref{iii3} that 
\begin{align*}
|B_t|
&<\frac{2\sigma a_t^2 \left(N(\Za) \lambda( N(\Za)) \right)^{1/2} |\Delta_t| }{\xi(\Za)^3}
+\frac{4a_t^2\sum_{z\in\mathcal{K}_t^+}\xi(z)^2}{\xi(\Za)^4}\\&
=2\sigma |\Delta_t| \cdot\frac{(N(\Za)\lambda(N(\Za))^{1/2}}{a_t}\cdot\frac{a_t^3}{\xi(\Za)^3} \\
& \qquad +4\cdot\frac{N(\Za)\lambda(N(\Za))}{a_t^2}\cdot
\frac{a_t^4}{\xi(\Za)^4}\cdot\frac{\sum_{z\in\mathcal{K}_t^+}\xi(z)^2}{N(\Za)\lambda(N(\Za))},
\end{align*} 
which is bounded in probability by \eqref{cltx}, Proposition~\ref{LB1}, Lemma~\ref{l:etacr} and \eqref{lln} for $\alpha>2$, \eqref{lln2} for $\alpha=2$. Hence we obtain
\begin{align*}
a_t^2\Big|\frac{1}{\Sigma_t^+}-\frac{1}{\bar{S}_t}\Big|&=a_t^2\Big|\Big(A_t^2+\frac{B_t}{a_t^2}\Big)^{-\frac 1 2}
-\frac{1}{\bar{S}_t}\Big|\\
&=a_t^2\Bigg|\frac{\xi(\Za)}{ \left(N(\Za)\lambda(N(\Za)) \right)^{1/2}}\Big[\Big(1+\frac{2\gamma }{\xi(\Za)}\Big)\Big(1+\frac{B_t}{A_t a_t^2}\Big)\Big]^{-\frac 1 2}\\&\phantom{=a_t^2\Bigg|}-\frac{\xi(\Za)}{ \left(N(\Za)\lambda(N(\Za)) \right)^{1/2}}+\frac{\gamma}{ \left(N(\Za)\lambda(N(\Za)) \right)^{1/2}}\Bigg|\\
&\le a_t^2\frac{\xi(\Za)}{ \left(N(\Za)\lambda(N(\Za)) \right)^{1/2}}\Big[
\frac{c_1}{\xi(\Za)^2}+\frac{c_2B_t}{A_t a_t^2}\Big]\\
&=c_1\cdot \frac{a_t}{\xi(\Za)}\cdot
\frac{a_t}{ \left(N(\Za)\lambda(N(\Za)) \right)^{1/2}}+\frac{c_2B_t}{A_t}\cdot
\frac{\xi(\Za)}{a_t}\cdot\frac{a_t}{ \left(N(\Za)\lambda(N(\Za)) \right)^{1/2}}
\end{align*}
with some positive constants $c_1$ and $c_2$ eventually almost surely. It remains to observe that the expression on the right-hand side is bounded in probability  by  Proposition~\ref{LB1} and Lemma~\ref{l:etacr}. 
\end{proof}

\subsection{Proof of Propositions \ref{p:e} and \ref{p:e2}}

The first five properties contained in the event $\mathcal{E}^ 1_t$ follow from Proposition~\ref{LB1}, Lemmas~\ref{l:os} and~\ref{l:gap}, and the symmetry of the model, while the final property is proven as in the proof of \cite[Proposition 5.6]{MPS}. The properties contained in the event $\mathcal{E}^2_t$ are a consequence of Corollary~\ref{c:kk}, and Lemmas~\ref{l:gap} and~\ref{l:nohigh}. 
\smallskip

The first property contained in the event $\mathcal{E}_t^{cr}$ follows from 
\begin{align*}
\lambda(r_t)\frac{M_t}{a_t}=\frac{N(\Za)}{\eta(\Za)}
\cdot\frac{\eta(\Za)}{\kappa(\Za)}\cdot\Big[\frac{\Za}{r_t}\Big]^{\frac{2}{\alpha}}\cdot \frac{a_t}{\xi(\Za)}\cdot\Big(1+\frac{\gamma}{\xi(\Za)}\Big)\cdot\frac{\lambda(r_t)}{\lambda(\Za)}
\Rightarrow\frac{\beta(X^{\ssup 1})^{2/\alpha}}{Y^{\ssup 1}}
\end{align*}
by Proposition~\ref{LB1}, Lemma~\ref{l:nn}, and since $\eta(n)\sim \beta\kappa(n)$. The fourth property follows from 
\begin{align*}
\frac{(\lambda(r_t))^{-1/2}}{\Sigma_t}&=\Big[\frac{\eta(\Za)}{N(\Za)}\Big]^{1/2}
\cdot\Big[\frac{\kappa(\Za)}{\eta(\Za)}\Big]^{1/2}
\cdot \Big[\frac{r_t}{\Za}\Big]^{1/\alpha}
\cdot \frac{\xi(\Za)}{a_t}
\cdot\Big(1-\frac{\gamma}{\xi(\Za)}\Big)\cdot\Big[\frac{\lambda(\Za)}{\lambda(r_t)}\Big]^{1/2}\\
&\Rightarrow \frac{Y^{\ssup 1}}{\sqrt{\beta}(X^{\ssup 1})^{1/\alpha}}
\end{align*}
by Proposition~\ref{LB1}, Lemma~\ref{l:nn}, and since $\eta(n)\sim \beta\kappa(n)$. The remaining properties follow from Lemma~\ref{l:cr} together with a symmetry argument to handle the potential on~$-E$.
\bigskip


\section{Fluctuation theory}
\label{sec:fluct}

In this section we complete the proof of the main results by establishing Proposition \ref{p:clt} which describes the fluctuations of the quantity $Q_t$. We make use of the point process machinery developed in Section~\ref{sec:pp} and a central limit theorem of Lindenberg type \cite[Thm.\ B1]{MPS}. \smallskip

\subsection{Proof of Proposition \ref{p:clt}}
We assume throughout this proof that event $\mathcal{E}_t$ holds, which is valid by Proposition~\ref{p:e}. We begin by establishing \eqref{var2} and \eqref{var3} on the asymptotic behaviour of $\text{Var}_{\mathcal{F}_t} Q_t$. Observe first that $Q_t(z), z\in\mathcal{K}_t,$ are i.i.d.\ conditionally on $\mathcal{F}_t$. 
For each $t>0$, $z\in\mathcal{K}_t$, and $n\in\{1,2\}$ we have, using the substitution $x=y\xi(\Za)$,
\begin{align*}
\mathrm{E}_{\mathcal{F}_t}Q_t^n(z)
&=\theta_t^{\alpha}\int_{\theta_t}^{\xi(\Za)}\Big[-\log\Big(1-\frac{x}{\xi(\Za)}\Big)\Big]^n\frac{\alpha}{x^{\alpha+1}}dx\\
&=\frac{\theta_t^{\alpha}}{\xi(\Za)^{\alpha}}\int_{\theta_t/\xi(\Za)}^{1}\big[-\log(1-y)\big]^n\frac{\alpha}{y^{\alpha+1}}dy.
\end{align*}
Since $\eta(r_t)=o(r_t)$ we have
\begin{align*}
\frac{\theta_t}{\xi(\Za)}
=\frac{a_t}{\xi(\Za)}\cdot\begin{cases}
a_t^{-1}& \to 0 \quad \text{ if }\eta(n)\ll \kappa(n)\text{ or }\eta(n)\sim \beta \kappa(n),\\
f_t\cdot \Big[\frac{\eta(r_t)}{r_t}\Big]^{\frac{1}{\alpha-2}}& \to 0 \quad \text{ if }\eta(n)\gg \kappa(n)\text{ and }\alpha>2,\\
\exp\left(-\frac{r_t}{\eta(r_t)f_t}\right)&\to0\quad\text{ if }\eta(n)\gg\kappa(n)\text{ and }\alpha=2,
\end{cases}
\end{align*}
as $\mathcal{E}_t$ holds. Using 
\begin{align*}
\int_s^1\big[-\log(1-y)\big]^n\frac{\alpha}{y^{\alpha+1}}dy
\sim \int_s^1\frac{\alpha}{y^{\alpha+1-n}}dy\sim \begin{cases}\frac{\alpha}{\alpha-n}s^{n-\alpha}&\text{ if }\alpha>2\text{ or }\alpha=2,\,n=1,\\
-2\log s&\text{ if }\alpha=2,\,n=2,
\end{cases}
\end{align*}
as $s\to 0$ we obtain, as $\mathcal{E}_t$ holds,
\begin{align}
\label{expec}
\mathrm{E}_{\mathcal{F}_t}Q_t^n(z)
&\sim \begin{cases}\displaystyle\frac{\alpha}{\alpha-n}\frac{\theta_t^n}{\xi(\Za)^n}&\text{ if }\alpha>2\text{ or }\alpha=2,\,n=1,\\\displaystyle
2\frac{\theta_t^2}{\xi(\Za)^2}\log\frac{\xi(\Za)}{\theta_t}&\text{ if }\alpha=2,\,n=2.
\end{cases}
\end{align}
This implies that
\begin{align}
\label{var4}
\text{\rm Var}_{\mathcal{F}_t} Q_t(z)\sim\begin{cases} \displaystyle\frac{\sigma^2\theta_t^2}{\xi(\Za)^2}&\mbox{ if }\alpha>2,\\\displaystyle2\frac{\theta_t^2}{\xi(\Za)^2}\log\frac{\xi(\Za)}{\theta_t}&\mbox{ if }\alpha=2.
\end{cases}
\end{align}
Hence, in the case $\alpha>2$, by Proposition~\ref{LB1} and Lemmas~\ref{l:nn}--\ref{l:kk},
\begin{align}
\label{var}
\text{\rm Var}_{\mathcal{F}_t}Q_t
&\sim 
\sigma^2 \frac{\theta_t^{\alpha}|\mathcal{K}_t|}{N(\Za)}\cdot\frac{N(\Za)}{\eta(\Za)} 
\left\{\begin{array}{lll}
\displaystyle
\frac{\eta(\Za)}{\eta(r_t)}\cdot
\frac{a_t^2}{\xi(\Za)^2}\cdot\frac{\eta(r_t)}{r_t^{2/\alpha}} & \! \stackrel{p}{\to} 0 & \! \! \! \text{ if }\eta(n)\ll \kappa(n),\\
\displaystyle
\frac{\eta(\Za)}{(\Za)^{2/\alpha}}\cdot
\frac{(\Za)^{2/\alpha}}{\xi(\Za)^2}
 & \!  \Rightarrow  2\beta\sigma^2B^2 & \! \! \! \text{ if }\eta(n)\sim \beta \kappa(n),\\
\displaystyle
\frac{\eta(\Za)}{\eta(r_t)}\cdot
\frac{a_t^2}{\xi(\Za)^2}\cdot f_t^{2-\alpha} & \! \stackrel{p}{\to}\infty & \! \! \! \text{ if }\eta(n)\gg \kappa(n),\\
\end{array}\right.
\end{align}
 where we have also used the fact that $|\mathcal{K}_t| = |\mathcal{K}^+_t| + |\mathcal{K}_t^-|$, and $\stackrel{p}{\to}$ denotes convergence in probability. In the case $\alpha=2$, we instead have
\begin{align}\label{var0}
\mathrm{Var}_{\mathcal{F}_t}Q_t&=2\frac{\theta_t^2|\mathcal{K}_t|}{N(\Za)}\cdot\frac{N(\Za)}{\eta(\Za)} \left\{\begin{array}{lll}
\displaystyle
\frac{\eta(\Za)\log\Za}{\Za}\cdot\frac{\Za}{\xi(\Za)^2}\cdot\frac{\log\xi(\Za)}{\log\Za}&\! \stackrel{p}{\to} 0& \! \! \mbox{ if }\eta(n)\ll\kappa(n),\\\displaystyle
\frac{\eta(\Za)\log\Za}{\Za}\cdot\frac{\Za}{\xi(\Za)^2}\cdot\frac{\log\xi(\Za)}{\log\Za}&\! \Rightarrow 2\beta B^2 & \! \! \mbox{ if }\eta(n)\sim\beta\kappa(n),\\\displaystyle
\frac{\eta(\Za)}{\eta(r_t)}\cdot\frac{r_t}{\xi(\Za)^2}\cdot\frac1{f_t}&\! \stackrel{p}{\to}\infty& \! \! \mbox{ if }\eta(n)\gg\kappa(n),
\end{array}\right.
\end{align}
where in the critical case we use the additional fact that, on the event $\mathcal{E}_t$,
 \[ \frac{\log\xi(\Za)}{\log\Za} \rightarrow \frac1{2}. \]
 This establishes \eqref{var2} and \eqref{var3}.

\smallskip
To prove~\eqref{clt}, observe that by construction $\mathrm{E}_{\mathcal{F}_t}V_t(z)=0$ for all $z\in\mathcal{K}_t$ and $\mathrm{E}_{\mathcal{F}_t}V_t^2=1$
almost surely. Then by~\cite[Thm.\ B1]{MPS}, 
~\eqref{clt} would follow from the Lindenberg condition 
\begin{align}
\label{tt4}
\sum_{z\in\mathcal{K}_t} \mathrm{E}_{\mathcal{F}_t}\big[V_t^2(z)\one\{|V_t(z)|\ge 2\e\}\big]
\to 0
\end{align}
in probability, for all $\e>0$. It suffices now to check that this condition is fulfilled in the critical and supercritical regimes. \smallskip

To do so, 
remark that, according to~\eqref{defv} 
\begin{align}
\label{kkk1}
Q_t(z)={\mathrm E}_{\mathcal{F}_t}Q_t(z)+V_t(z)\sqrt{\text{Var}_{\mathcal{F}_t}Q_t}.
\end{align}
Further, by~\eqref{expec}, \eqref{var} and \eqref{var0}
we have by Lemma~\ref{l:kk},
\begin{align}
\label{kkk3}
\frac{|\mathrm{E}_{\mathcal{F}_t} Q_t(z)|}{\sqrt{\text{Var}_{\mathcal{F}_t}Q_t}} < \frac{C}{\sqrt{|\mathcal{K}_t|}} \to 0,
\end{align}
in probability, for some constant $C > 0$. Since $Q_t(z)$ and ${\mathrm E}_{\mathcal{F}_t}Q_t(z)$
are both non-negative almost surely, it follows from~\eqref{kkk1} and~\eqref{kkk3}  that 
\begin{align*}
\big\{|V_t(z)| \ge 2\varepsilon\big\} \subseteq 
\big\{|Q_t(z)|\ge \e\sqrt{\text{Var}_{\mathcal{F}_t}Q_t}\big\}.
\end{align*}
eventually with overwhelming probability. Hence 
\begin{align}
\sum_{z \in \mathcal{K}_t} {\mathrm E}_{\mathcal{F}_t} 
\big[V_t^2(z) & \one{\{ |V_t(z)| \ge 2\varepsilon \}} \big]  
 \le  \frac{1}{\text{Var}_{\mathcal{F}_t}Q_t}
\sum_{z \in \mathcal{K}_t} {\mathrm E}_{\mathcal{F}_t} 
\Big[\big(Q_t(z)-{\mathrm E}_{\mathcal{F}_t}Q_t(z)\big)^2\one\big\{ |Q_t(z)| \ge \varepsilon \sqrt{\text{Var}_{\mathcal{F}_t}Q_t}\big\}\Big] \notag\\
&=\frac{1}{\text{Var}_{\mathcal{F}_t}Q_t(z)}
 {\mathrm E}_{\mathcal{F}_t} 
\Big[\big(Q_t(z)-{\mathrm E}_{\mathcal{F}_t}Q_t(z)\big)^2\one\big\{ |Q_t(z)| \ge \varepsilon \sqrt{\text{Var}_{\mathcal{F}_t}Q_t}\big\}\Big]
\label{tt1}
\end{align}
for any $z\in\mathcal{K}_t$.
Denoting $\nu_t^{\e}=\exp\big\{-\varepsilon \sqrt{\text{Var}_{\mathcal{F}_t}Q_t}\big\}$, we have
\begin{align}
\label{mmm1}
\big\{ |Q_t(z)| \ge \varepsilon \sqrt{\text{Var}_{\mathcal{F}_t}Q_t}\big\}
=\big\{(1-\nu_t^{\e})\xi(\Za)\le\xi(z)<\xi(\Za)\big\}.
\end{align}
In the critical regime by~\eqref{var3} we have $(1-\nu_t^{\e})\xi(\Za)>\theta_t$ eventually.  
In the supercritical regime~\eqref{var2} implies that  
$\nu_t^{\e}\to 0$, and hence since $\eta(r_t)\ll r_t$ and by Proposition~\ref{LB1} we have 
\begin{align*}
\frac{(1-\nu_t^{\e})\xi(\Za)}{\theta_t}
=(1-\nu_t^{\e})\cdot \frac{\xi(\Za)}{a_t}\cdot\begin{cases} 
\frac1{f_t}\Big[\frac{r_t}{\eta(r_t)}\Big]^{\frac{1}{\alpha-2}}&\stackrel{p}{\to}\infty\quad\mbox{if }\alpha>2,\\
\exp\left(\frac{r_t}{\eta(r_t)f_t}\right)&\stackrel{p}{\to}\infty\quad\mbox{if }\alpha=2,
\end{cases}
\end{align*}
and therefore also $(1-\nu_t^{\e})\xi(\Za)>\theta_t$ eventually. 
\smallskip

We use the change of variables $x=y\xi(\Za)$ and~\eqref{mmm1} to compute, for $n\in\{0,1,2\}$,
\begin{align*}
\mathrm{E}_{\mathcal{F}_t}&\Big[Q_t^n(z)\one\big\{|Q_t(z)|\ge \varepsilon \sqrt{\text{Var}_{\mathcal{F}_t}Q_t}\big\}\Big]\\
&=\theta_t^{\alpha}\int_{\theta_t}^{\xi(\Za)}
\Big[-\log\Big(1-\frac{x}{\xi(\Za)}\Big)\Big]^n
\one\big\{|Q_t(z)|\ge \varepsilon \sqrt{\text{Var}_{\mathcal{F}_t}Q_t}\big\}
\frac{\alpha}{x^{\alpha+1}}dx\notag\\
&= \theta_t^{\alpha}\int_{(1-\nu_t^{\e})\xi(\Za)}^{\xi(\Za)}
\Big[-\log\Big(1-\frac{x}{\xi(\Za)}\Big)\Big]^n \frac{\alpha}{x^{\alpha+1}}dx
=\frac{\theta_t^{\alpha}}{\xi(\Za)^{\alpha}}J_n(\nu_t^{\e}).
\end{align*}
where, for any $x\in [0,1]$,  
\begin{align*}
J_n(x)=\int_{1-x}^{1} 
[-\log(1-y)]^n \frac{\alpha}{y^{\alpha+1}}dy.
\end{align*}
Observe that in the critical regime $J_n(\nu_t^{\e})$ is bounded away from zero and infinity in probability by~\eqref{var3}, 
and in the supercritical regime $J_n(\nu_t^{\e})$ tends to zero and hence is also bounded by~\eqref{var2}.  Further, recall from \eqref{expec},
\begin{align*}
\mathrm{E}_{\mathcal{F}_t}Q_t(z)\sim  
\left\{\begin{array}{lll}
\displaystyle
\frac{\alpha}{\alpha-1}\xi(\Za)^{-1} & \to 0 & \text{ if }\eta(n)\sim \beta\kappa(n),\\
\displaystyle
\frac{\alpha}{\alpha-1}f_t\Big[\frac{a_t^{\alpha-2}}{\xi(\Za)^{\alpha-2}}
\cdot\frac{\eta(r_t)}{r_t}\Big]^\frac{1}{\alpha-2} & \to 0 &  \text{ if }\eta(n)\gg \kappa(n)\mbox{ and }\alpha>2,\\
\displaystyle
2\frac{a_t}{\xi(\Za)}\exp\left(-\frac{r_t}{\eta(r_t)f_t}\right)&\to0&\mbox{ if }\eta(n)\gg\kappa(n)\mbox{ and }\alpha=2,
\end{array}\right.
\end{align*}
by Proposition~\ref{LB1}. Expanding the square in~\eqref{tt1} and using~\eqref{var4} we obtain that for any $z\in\mathcal{K}_t$, eventually with overwhelming probability
\begin{align*}
\sum_{z \in \mathcal{K}_t} {\mathrm E}_{\mathcal{F}_t} 
\big[V_t^2(z)\one{\{ |V_t(z)| \ge 2\varepsilon \}} \big]  
&\le \frac1{\text{Var}_{\mathcal{F}_t}Q_t(z)}\mathrm{E}_{\mathcal{F}_t}\Big[ Q_t^2(z)\one\big\{|Q_t(z)|\ge \varepsilon \sqrt{\text{Var}_{\mathcal{F}_t}Q_t}\big\}\\&\phantom{\le\frac1{\text{Var}_{\mathcal{F}_t}Q_t(z)}\mathrm{E}_{\mathcal{F}_t}\Big[}+(\mathrm{E}_{\mathcal{F}_t}Q_t(z))^2\one\big\{|Q_t(z)|\ge \varepsilon \sqrt{\text{Var}_{\mathcal{F}_t}Q_t}\big\}\Big]
\\&=\frac{\theta_t^\alpha\xi(\Za)^{-\alpha}}{\text{Var}_{\mathcal{F}_t}Q_t(z)}\left[J_2(\nu_t^\varepsilon)+J_0(\nu_t^\varepsilon)o(1)
\right]\\
&<\begin{cases}
\displaystyle \frac{2}{\sigma^2}\xi(\Za)^{2-\alpha}&\stackrel{p}{\to}0\quad\mbox{ if }\eta(n)\sim\beta\kappa(n),\,\alpha>2,\\\displaystyle 
\frac{2}{\sigma^2}f_t^{\alpha-2}\cdot\frac{\eta(r_t)}{r_t}\left(\frac{\xi(\Za)}{a_t}\right)^{2-\alpha}&\stackrel{p}{\to}0\quad\mbox{ if }\eta(n)\gg\kappa(n),\,\alpha>2,\\\displaystyle 
(\log\xi(\Za))^{-1}&\stackrel{p}{\to}0\quad\mbox{ if }\eta(n)\sim\beta\kappa(n),\,\alpha=2,\\
\displaystyle \left(\log\frac{\xi(\Za)}{a_t}+\frac{r_t}{\eta(r_t)f_t}\right)^{-1}&\stackrel{p}{\to}0\quad\mbox{ if }\eta(n)\gg\kappa(n),\,\alpha=2,
\end{cases}
\end{align*}
by Proposition~\ref{LB1}, as required in \eqref{tt4}. 

\bigskip


\printbibliography

\end{document}